\documentclass[12pt,final]{amsart}
\usepackage[asymmetric, hmarginratio=2:3]{geometry}
\usepackage[colorlinks=true, citecolor=blue, linkcolor=blue, urlcolor=blue]{hyperref}
\usepackage{upref}
\usepackage{amssymb}

\def\d{~\mbox{d}}

\def\T{{\mathcal T_h}}

\usepackage{stmaryrd}
\usepackage{mathtools}
\DeclarePairedDelimiter\jump{\llbracket}{\rrbracket}
\DeclarePairedDelimiter\tnorm{\vert\hspace{-0.3mm}\Vert}{\vert\hspace{-0.3mm}\Vert}
\DeclarePairedDelimiter\floor{\lfloor}{\rfloor}

\def\p{\partial}
\DeclareMathOperator{\supp}{supp}
\DeclareMathOperator{\grad}{grad}
\DeclareMathOperator{\diam}{diam}

\def\R{\mathbb R}

\def\N{\mathbb N}
\def\Np{\mathbb N^+}

\def \pp {\mathbf{p}}

\DeclarePairedDelimiter\pair{\langle}{\rangle}
\DeclarePairedDelimiter\norm{\lVert}{\rVert}

\let\div\relax
\DeclareMathOperator{\div}{div}

\usepackage{ifdraft}
\ifoptionfinal{
\usepackage[disable]{todonotes}
}{
\date{Compiled \today}
\usepackage[norefs, nocites]{refcheck}
\usepackage[breakall, fit]{truncate}
\usepackage{accsupp}
\setlength{\marginparsep}{5pt}

\usepackage[notref, notcite]{showkeys}
\usepackage[bordercolor=white, color=white, textwidth=100pt]{todonotes}
}
\newcommand{\HOX}[1]{\todo[noline, size=\footnotesize]{#1}}

\makeatletter\providecommand\@dotsep{5}\def\listtodoname{List of Todos}\def\listoftodos{\hypersetup{linkcolor=black}\@starttoc{tdo}\listtodoname\hypersetup{linkcolor=blue}}\makeatother

\usepackage[colored]{shadethm} 
\definecolor{shaderulecolor}{rgb}{0.651,0.074,0.090}

\newshadetheorem{theorem}{Theorem}[section]
\newshadetheorem{lemma}[theorem]{Lemma}
\newtheorem{proposition}[theorem]{Proposition}

\theoremstyle{definition}

\theoremstyle{remark}
\newshadetheorem{remark}[theorem]{Remark}

\title[Spacetime FEMs for control problems]{Spacetime finite element methods for control problems subject to the wave equation}
\author[Burman]{Erik Burman}
\address{Department of Mathematics, University College London, Gower Street, London UK, WC1E 6BT.}
\email{e.burman@ucl.ac.uk}
\thanks{EB acknowledges funding by EPSRC grants EP/P01576X/1}
\author[Feizmohammadi]{Ali Feizmohammadi}
\address{Fields institute, 222 College St, Toronto, Canada, M5T 3J1}
\email{afeizmoh@fields.utoronto.ca}
\thanks{AF acknowledges funding by EPSRC grant EP/P01593X/1 and support from the Fields institute for research in mathematical sciences."
}
\author[M\"{u}nch]{Arnaud M\"unch}
\address{Laboratoire de Math\'ematiques Blaise Pascal, Universit\'e Clermont Auvergne, UMR CNRS 6620, Campus des C\'ezeaux,  63177~Aubi\`ere, France}
\email{arnaud.munch@uca.fr}
\thanks{AM acknowledges funding by the French government research program ``Investissements d'Avenir'' through the IDEX-ISITE initiative 16-IDEX-0001 (CAP 20-25)}
\author[Oksanen]{Lauri Oksanen}
\address{Department of Mathematics and Statistics, University of Helsinki, P.O 68, 00014 University of Helsinki, Finland}
\email{lauri.oksanen@helsinki.fi}
\thanks{LO acknowledges funding by EPSRC grants EP/P01593X/1 and EP/R002207/1}

\begin{document}
\begin{abstract}
We consider the null controllability problem for the wave equation, and analyse a stabilized finite element method formulated on a global, unstructured
spacetime mesh. We prove error estimates for the approximate control given by the computational method. The proofs are based on the regularity properties of the control given by the Hilbert Uniqueness Method,
together with the stability
properties of the numerical scheme. Numerical experiments illustrate the results. 
\end{abstract}
\maketitle
\ifoptionfinal{}{
\tableofcontents
}

\section{Introduction}

We consider the classical null controllability problem for the wave equation, both with distributed and boundary control. 
Let $T>0$, and let $\Omega \subset \R^n$, with $n \ge 2$, be a connected
bounded open set with smooth boundary.
We write 
    \begin{align*}
M = (0,T) \times \Omega, \quad \Gamma = (0,T) \times \p \Omega.
    \end{align*}

For a fixed initial state $(u_0, u_1) \in H_0^1(\Omega) \times L^2(\Omega)$, 
the distributed null control problem on $M$ reads: find $\phi \in L^2(M)$ such that the solution $u$ of 
    \begin{align}\label{wave_eq_intro}
\begin{cases}
\Box u = \chi \phi, 
\\
u|_\Gamma =0,
\\
u|_{t=0} = u_0,\ \p_t u|_{t=0} = u_1,
\end{cases}
    \end{align}
satisfies
    \begin{align}\label{null_target}
u|_{t=T} = 0,\quad \p_t u|_{t=T} = 0.
    \end{align}
Here $\Box = \p_t^2 - \Delta$ is the wave operator, and $\chi$ is a cutoff function that localizes the control in a subset of $M$. 
More precisely, we consider a cutoff of the form
    \begin{align*}
\chi(t,x) = \chi_0(t) \chi_1^2(x)
    \end{align*}
where $\chi_0 \in C_0^\infty([0,T])$
and $\chi_1 \in C^\infty(\Omega)$ take values in $[0,1]$.

Our main assumption is that 
\begin{itemize}
\item[(A)] $\chi = 1$
on open $(a,b) \times \omega \subset M$
satisfying the geometric control condition.
\end{itemize}
The geometric control condition
means that every compressed generalized bicharacteristic intersects the set $(a,b) \times \omega$, when projected to $M$.
We refer to \cite{BLRII} for the rather technical definition of a compressed generalized bicharacteristic.
Roughly speaking, all continuous paths on $M$, consisting of lightlike line segments in its interior and reflected on $\Gamma$ according to Snell's law, 
must intersect $(a,b) \times \omega$. However, projections of compressed generalized bicharacteristics may also glide along $\Gamma$ under suitable convexity.

For our main result we assume, furthermore, that $(u_0,u_1) \in H^{k+1}(\Omega)\times H^{k}(\Omega)$ for some $k = 2,3, \dots$, and that the following compatibility conditions of order $k$ are satisfied
    \begin{align}
    \label{C1L}
    \tag{C1}
(\Delta^j u_0)|_{\p \Omega}=0
\quad &\text{for} \quad
j=0,1,\ldots,\floor*{\frac{k}{2}},
\\
\label{C2L}
\tag{C2}
(\Delta^j u_1)|_{\p \Omega}=0
\quad &\text{for} \quad
j=0,1,\ldots,\floor*{\frac{k-1}{2}}.
    \end{align}
Here $\floor{\cdot}$ is the floor function that gives the greatest integer less than or equal to its argument. 
We recall that the compatibility conditions guarantee that, for smooth enough $\phi$, the solution $u$ of (\ref{wave_eq_intro}) is in $H^{k+1}(M)$, see e.g. \cite[Theorem 6, p. 412]{EvansPDEbook}.

Under these assumptions, we show that the stabilized finite element method introduced below gives such an approximation $\phi_h$ of a certain minimum norm solution $\phi$ to the control problem that 
    \begin{align}\label{conv_intro}
\norm{\chi(\phi_h - \phi)}_{L^2(M)} \lesssim h^{q},
    \end{align}
where $h>0$ is the mesh size and $q \le k-1$ is the polynomial order of the finite element space.
The implicit constant in the above inequality is independent of $h$ and the functions $\phi_h$ and $\phi$. This notation is used in the paper when confusion is not likely to arise.
See Theorem \ref{th:res_conv} for the precise formulation.
In this result both $u$ and $\phi$ are assumed to be at least $H^2(M)$-smooth, corresponding to the above constraint $k \ge 2$. 
We prove also a weak convergence result for our method 
assuming only the smoothness $u_1 \in L^2(\Omega)$
in the case that $u_0 = 0$, see Theorem \ref{th_rough}. The 
case with general rough data is left for future work.

Let us now sketch our result in the case of the boundary null control problem of the following form:
given an initial state $(u_0, u_1) \in H_0^1(\Omega) \times L^2(\Omega)$, 
find $\psi \in L^2(\Gamma)$ such that the solution $u$ of 
    \begin{align}\label{wave_eq_intro_bd}
\begin{cases}
\Box u = 0,
\\
u|_\Gamma = \chi \psi,
\\
u|_{t=0} = u_0,\ \p_t u|_{t=0} = u_1,
\end{cases}
    \end{align}
satisfies the final time condition (\ref{null_target}).
Under a geometric control condition analogous to (A), and the above regularity assumptions on the data $(u_0, u_1)$,
we introduce a finite element method that converges as
    \begin{align}\label{conv_intro_bd}
\norm{\chi(\psi_h - \psi)}_{L^2(\Gamma)} \lesssim h^{q-\frac12},
    \end{align}
with notations analogous to (\ref{conv_intro})
and $q \le k+1$.
See Theorem \ref{th:res_conv_bd} for the precise formulation.
Observe that, although there is a loss of order $1/2$ in (\ref{conv_intro_bd}) in comparison to (\ref{conv_intro}), 
if the highest possible polynomial orders are used,
the order in (\ref{conv_intro_bd}) becomes $k+1/2$
versus $k-1$ in (\ref{conv_intro}).
We can also rescale the method in the distributed case to get the order $q-1/2$ for $q \le k$, leading to $k-1/2$ for the highest possible order.

\subsection{Literature }

This work is a contribution to the finite dimensional approximation of null controls for the linear wave equation. The seminal work is due to Glowinski and Lions  
in \cite{Glowinski_Japan_1990} where the search of the control of minimal $L^2$ norm is reduced (using the Fenchel--Rockafellar duality theory) to the unconstrained minimization of the corresponding conjugate functional involving the homogeneous adjoint problem. Minimization of the discrete functional, associated with centered finite difference approximation in time and $P^1$ finite element method in space is discussed at length in \cite{Glowinski_Japan_1990}
and exhibits a lack of convergence of the approximation with respect to the discretization parameter $h$. This is due to spurious 
high frequencies discrete modes which are not exactly controllable uniformly in $h$. 

This pathology can easily be avoided in practice by adding to the conjugate functional a regularized Tikhonov parameter,
but this leads to so called approximate controls, solving the control problem only up to a small remainder term. Several cures aiming to filter out the high frequencies  have been proposed and analyzed, mainly for simple geometries (1d interval, unit square in 2d, etc) with finite differences schemes. The simplest, but artificial, approach is to eliminate the highest eigenmodes of a discrete approximation of the initial condition as discussed in one space dimension in \cite{micu_NM_2002}, and extended in \cite{lissy_roventa_2019}. We mention spectral methods initially developed \cite{bourquin1993} then used in \cite{lebeau_nodet_2010}. We also mention so called bi-grid method (based on the projection of the discrete gradient on a coarse grid) proposed in \cite{Glowinski_Japan_1990} and analyzed in \cite{Loreti_Mehrenberger_2008, ignat_zuazua_2009} leading to convergence results. One may also design more elaborated discrete schemes avoiding spurious modes: we mention \cite{Glowinski_Wheeler_1989} based on a mixed reformulation of the wave equation analyzed later with finite difference schemes in \cite{castro_micu_2006, castro_micu_munch_2008, asch_munch_2009} at the semi-discrete level and then extended in \cite{munch_m2an2005} to a full space-time discrete setting, leading to strong convergent results. 

The above previous works, notably reviewed in
\cite{Zuazua_review_2005,ervedoza_zuazua_survey_2013}, fall within an
approach that can be called ``discretize then control'' as they aim to
control exactly to zero a finite dimensional approximation of the wave
equation. A relaxed controllability approach is analyzed in
\cite{BFO_Sicon2020} using a stabilized finite element method in space
and leading for smooth two and three dimensional geometries to a
strong convergent approximation. The controllability requirement is
imposed via appropriate penalty terms. We also mention
\cite{pedregal_periago_jorge_2008} based on the Russel principle,
extended in \cite{cindea_micu_tucsnak_2011} and
\cite{gunzburger_2006,aranda_pedregal_2014} for least-squares based
method. One the other hand, one may also employ a ``control then
discretize'' procedure, where the optimality system (for instance
associated with the control of minimal $L^2$ norm ) mixing the
boundary condition in time and space and involving the primal and
adjoint state is discretized within a priori a space-time
approximation. The well-posedness of such system is achieved by using
so called global or generalized observability inequalities. Such
approach avoids the numerical pathologies mentioned above and is
notably well-suited for mesh adaptivity. On the other hand, the
numerical analysis, within a conformal approximation is delicate since
it requires to prove inf-sup stablity that is uniform with respect to $h$. We mention \cite{cindea_munch_calcolo2015} where this approach has been introduced within a conformal approximation leading to convergent numerical results for the control of minimal $L^2$ norm. It has been extended in \cite{montaner_munch_2019} where the wave equation is reformulated as a first order system, solved in the one dimensional case with a stabilized formulation allowing to bypass the inf-sup property issue.  We also mention \cite{cindea_efc_munch_2013}
in the 1d case where the optimality system associated to cost involving both the control and the state is reformulated as a space-time elliptic problem of order four, leading to strong convergent result with respect to the discretization parameter.   
The present paper falls into this category and aims, in the spirit of
\cite{BFMO_2021} devoted to the dual data assimilation problem, to
provide some convergent results, including rate of convergence, with
respect to the discrete parameter. We mention a growing interest for
space-time (finite element) methods of approximation for the wave
equation, initially advocated in \cite{hughes_spacetime_1990, Fre93, John93} and more recently in \cite{+2019}, \cite{Mazzieri_2020}, \cite{dumont_2018}, \cite{Wieners2016}, \cite{Steinbach2019}.

\section{Distributed control}

We recall that, assuming (A),
the distributed control problem can be solved by finding $u$ and $\phi$ such that 
    \begin{align}\label{control_eqs}
\begin{cases}
\Box u = \chi \phi, 
\\
u|_{x \in \p \Omega}=0,
\\
u|_{t=0} = u_0,\ \p_t u|_{t=0} = u_1,
\\
u|_{t=T} = 0,\ \p_t u|_{t=T} = 0,
\end{cases}
\qquad
\begin{cases}
\Box \phi = 0, 
\\
\phi|_{x \in \p \Omega}=0.
\end{cases}
    \end{align}
Moreover, if $(u_0,u_1) \in H^{k+1}(\Omega)\times H^{k}(\Omega)$
satisfies the compatibility conditions of order $k$,
then the unique solution $(u,\phi)$ to (\ref{control_eqs}) satisfies 
    \begin{align}\label{phi_init}
\phi|_{t=T} \in H^{k}(\Omega), 
\quad 
\p_t \phi|_{t=T} \in H^{k-1}(\Omega), 
    \end{align}
and this initial data for $\phi$ satisfies the compatibility conditions of order $k-1$, see \cite[Theorem 5.1]{Ervedoza2010}.
It follows that $\phi \in H^k(M)$, and this again implies that $u \in H^{k+1}(M)$.
The convergence proof for our finite element method is based on the fact that the solution of (\ref{control_eqs}) has this regularity.

The control given by $\phi$ can be characterized also as the control with the minimum norm on $M$ with respect to the weighted measure $\chi dt dx$.
The fact that (\ref{control_eqs}) has a unique solution follows from this characterization, however, we give a short independent proof for the convenience of the reader. 

\begin{lemma}\label{lem_uniq}
Suppose that (A) holds.
Let $u, \phi \in L^2(M)$ solve \eqref{control_eqs} with $u_0 = u_1 = 0$.
Then $u = \phi = 0$.
\end{lemma} 
\begin{proof}
The lateral boundary traces on $(0,T) \times \p \Omega$ are well-defined due to partial hypoellipticity, see Lemma \ref{lem_parthypo} in Appendix \ref{appendix_cont}.
Typical energy estimates, see e.g. \cite{LLT}, give
    \begin{align*}
u \in C(0,T; H^1_0(\Omega)) \cap C^1(0,T; L^2(\Omega)),
    \end{align*}
and Lemma \ref{lem_linfty_ptu} in Appendix \ref{appendix_cont} implies that
    \begin{align*}
\phi \in C(0,T; L^2(\Omega)) \cap C^1(0,T; H^{-1}(\Omega)).
    \end{align*}
In particular, we may parametrize $\phi$
by $\phi|_{t=0}$ and $\p_t \phi|_{t=0}$.
Let $\phi_{j}^0, \phi_{j}^1 \in C_0^\infty(\Omega)$ 
satisfy 
$\phi_{j}^0 \to \phi|_{t=0}$ in $L^2(\Omega)$ 
and
$\phi_{j}^1 \to \p_t \phi|_{t=0}$ in $H^{-1}(\Omega)$.
Write $(u_j, \phi_j)$ for the solution of
    \begin{align*}
\begin{cases}
\Box u = \chi \phi, 
\\
u|_{x \in \p \Omega}=0,
\\
u|_{t=0} = 0,\ \p_t u|_{t=0} = 0,
\end{cases}
\qquad
\begin{cases}
\Box \phi = 0, 
\\
\phi|_{x \in \p \Omega}=0.
\\
\phi|_{t=0} = \phi_j^0,\ \p_t \phi|_{t=0} = \phi_j^1.
\end{cases}
    \end{align*}
Then 
    \begin{align*}
(\chi \phi_j, \phi_j)_{L^2(M)} 
&= 
(\Box u_j, \phi_j)_{L^2(M)} - (u_j, \Box \phi_j)_{L^2(M)}
\\&= (\p_t u_j|_{t=T}, \phi_j|_{t=T})_{L^2(\Omega)}
- (u_j|_{t=T}, \p_t \phi_j|_{t=T})_{H^1_0 \times H^{-1}(\Omega)}.
    \end{align*}
Taking the limit $j \to \infty$ shows that $\phi = 0$ in $\supp(\chi)$.
The distributed observability estimate, see Theorem \ref{th_obs_dist}, implies that $\phi = 0$ in $M$. It follows that also $u = 0$ in $M$.
\end{proof} 

\subsection{Notations}
We write $(t,x) = (x^0, x^1, \dots, x^n)$ for the coordinates on $\R^{1+n}$. Let $g$ stand for the Minkowski metric on $\R^{1+n}$, and
denote by $g(\cdot, \cdot)$ the scalar product with respect to $g$.
The wave operator can be written as $\Box = -\div \grad u$, where the divergence and gradient are defined with respect to $g$. 
Let $K \subset M$ be an open set with piecewise smooth boundary,
and let $N = (N_0, \dots, N_n)$ be the outward pointing unit normal vector field on $\p K$,
defined with respect to the Euclidean metric on $\R^{1+n}$.
We write 
    \begin{align*}
\p_\nu u = N \cdot \grad v = -N_0 \p_{x^0} v + N_1 \p_{x^1} v + \dots + N_n \p_{x^n} v.
    \end{align*}
Note that $\div$ coincides with the Euclidean divergence, and we can apply the Euclidean divergence theorem to obtain
    \begin{align}\label{int_by_parts}
\int_{K} u \Box v \d x = 
\int_K g(du, dv) \d x 
- \int_{\p K} u \p_\nu v \d s,
    \end{align}
where $\mbox{d}s$ is the Euclidean surface measure on $\p K$,
and $d u$ is the spacetime differential of $u$, that is, the covector with the components $\p_{x^j} u$, $j=0, \dots, n$.

\subsection{Discretization}
\label{sec_disc}

Consider a family $\mathcal T = \{\mathcal{T}_h : h > 0\}$ 
where $\mathcal{T}_h$ is a set of $1+n$-dimensional simplices forming a
simplicial complex. 
To keep the discussion as simple as possible, we assume in this section that 
$\bigcup_{K \in \mathcal{T}_h} = M$ for all $h > 0$.
This is a restrictive assumption since we also assumed that the spatial boundary $\p \Omega$ is smooth.  
We will explain later, see Remark \ref{rem_dist_bd}, how this issue can be avoided by allowing the simplices adjacent to the boundary to have curved faces, fitting $\Omega$. This fitting technique is also described in detail in the context of the boundary control problem below. 

If the set $\omega$ in assumption (A) is a neighbourhood of the boundary $\p \Omega$ then the distributed observability estimate in Theorem \ref{th_obs_dist} holds in the case of piecewise smooth $\p \Omega$ and large enough $T>0$. In particular, we can consider polyhedral $\Omega$ and then $\bigcup_{K \in \mathcal{T}_h} = M$ is straightforward to arrange. 
The multiplier method can also be used to derive the distributed observability estimate for polyhedral $\Omega$ and more general observation regions $\omega$, however, this method can not reproduce the sharp geometric control condition in the case of smooth boundary \cite{Miller2002}. 

We assume that the family $\mathcal{T}$ is quasi uniform,
see e.g. \cite[Definition 1.140]{Ern2004}, and
indexed by 
$$
h = \max_{K \in \T} \diam(K).
$$ 
Then we define for $p \in \Np = \{1,2,\dots\}$ the $H^1(M)$-conformal approximation space of polynomial degree $p$,
 \begin{align}
\label{def_Vh0}
V^p_{h} = \{u \in H^1(M) : u|_{\Gamma} = 0,\, u |_K \in \mathbb{P}_p(K)
\text{ for all $K \in \mathcal T_h$} \},
\end{align}
where $\mathbb{P}_p(K)$ denotes the set of polynomials of degree less
than or equal to $p$ on $K$. 
Occasionally we write also $V_h = \bigcup_{p \in \Np} V_h^p$.

For any $h>0$, the control problem (\ref{control_eqs}) can be formulated weakly as 
    \begin{align}\label{weak_form}
a(u,\psi) = h^2 c(\phi, \psi) + L(\psi), \quad a(v,\phi) = 0,
    \end{align}
for all $v,\psi \in C^\infty(M)$ vanishing on $\Gamma$, where
    \begin{align}\label{def_a1}
a(u,\psi) &= \int_{M} g(hdu, hd\psi) \d x 
-h(u, h\p_\nu \psi)_{L^2(\p M \setminus \Gamma)}
\\\notag
L(\psi) &= h(hu_1, \psi|_{t=0})_{L^2(\Omega)} - h(u_0, h\p_t \psi|_{t=0})_{L^2(\Omega)},
    \end{align}
and $c(\phi,\psi) = (\chi \phi, \psi)_{L^2(M)}$.
Indeed, it follows from (\ref{int_by_parts}) that if smooth $(u,\phi)$ solves (\ref{control_eqs})
then (\ref{weak_form}) holds for all smooth $(v,\psi)$ vanishing on $\Gamma$. 

The bilinear form $a$ is scaled so that 
there is $C > 0$ such that for all $h > 0$ and $u, v\in H^2(M) + V_h$ there holds
    \begin{align}\label{bilinear_bound}
a(u,v) \le C \norm{u}_{H^2(\mathcal T_h)}
\norm{v}_{H^2(\mathcal T_h)},
    \end{align}
where the broken semiclassical Sobolev norm is defined for any $k \in \N$ by 
    \begin{align*}
\norm{u}_{H^k(\mathcal T_h)}^2
= 
\sum_{j=0}^k \sum_{K \in \mathcal T_h} \norm{(h D)^j u}_{L^2(K)}^2.
    \end{align*}
Here $D^j u$ is the tensor of order $j$ that gives the $j$th total derivative of $u$. 
The continuity (\ref{bilinear_bound})
is consequence of the following trace inequality, see e.g. \cite[Eq. 10.3.9]{BS08}:
there is $C > 0$ such that for all $h > 0$, $K \in \mathcal T_h$ and $u \in H^1(K)$ there holds
\begin{align}
\label{trace_cont}
h^\frac12 \|u\|_{L^2(\partial K)} &\le C( \|u\|_{L^2(K)} + 
\|h\nabla u\|_{L^2(K)}).
\end{align}

For $u \in H^k(M)$ the broken semiclassical norm $\norm{u}_{H^k(\mathcal T_h)}$ reduces to the usual semiclassical norm defined by
    \begin{align*}
\norm{u}_{H_h^k(M)}^2
= 
\sum_{j=0}^k \norm{(h D)^j u}_{L^2(M)}^2.
    \end{align*}
Moreover, there is $C > 0$ such that for all $h > 0$ and $u \in V_h$ there holds
    \begin{align*}
\norm{u}_{H^k(\mathcal T_h)} \le C \norm{u}_{L^2(M)}.
    \end{align*}
This is due to the discrete inverse inequality, see e.g. \cite[Lemma 1.138]{Ern2004}: 
there is $C > 0$ such that for all $h > 0$, $K \in \mathcal T_h$, $p \in \Np$ and $u \in \mathbb{P}_p(K)$ there holds
    \begin{align}\label{inverse_disc}
\|h\nabla u\|_{L^2(K)} &\le C
\|u\|_{L^2(K)}.
    \end{align}
We will systematically use a scaling so that all the bilinear forms in the paper satisfy the bound (\ref{bilinear_bound}).

Our finite element method has the form:
find the critical point of the Lagrangian
    \begin{align*}
\mathcal L(u,\phi) : V_h^p \times V_h^q \to \R, \quad 
\mathcal L(u,\phi) = 
\frac12 h^2 c(\phi, \phi) + L(\phi)
- \frac12 \mathcal R(u,\phi)
- a(u, \phi),
    \end{align*}
where, writing $U = (u, \p_t u)$ and $U_0 = (u_0, u_1)$,
the regularization is given by
    \begin{align}\label{def_R}
\mathcal R(u,\phi)
&= h^{-\kappa} S(u) - h^\kappa S(\phi) 
+ h^{-\kappa} E(U|_{t=0} - U_0)
+ h^{-\kappa} E(U|_{t=T}) 
\\\notag&\qquad
+ h^{4-\kappa} \widetilde C(\phi) + 2 h^{2-\kappa}\rho(u, \phi),
\\\notag
E(U_0) &= 
h\norm{u_0}_{L^2(\Omega)}^2
+ 
h\norm{h u_1}_{L^2(\Omega)}^2,
\\\notag
S(u) &=
\sum_{K \in \mathcal T_h} \norm{h^2\Box u}_{L^2(K)}^2
+ \sum_{F \in \mathcal F_h} h \norm{\jump{h\p_\nu u}}_{L^2(F)}^2,
\\\notag
\widetilde C(\phi) &= 
\norm{\chi \phi}_{L^2(M)}^2,
\quad
\rho(u,\phi) = - \sum_{K \in \mathcal T_h} (h^2 \Box u, \chi \phi)_{L^2(K)},
    \end{align}


where $\kappa < 2$ is a fixed constant.
We have $\mathcal R(u,\phi) = 0$ for a smooth solution $(u,\phi)$ to 
(\ref{control_eqs}).
Indeed, 
    \begin{align*}
S(u) + 2 h^{2}\rho(u, \phi) + h^{4} \widetilde C(\phi)
= 
\sum_{K \in \mathcal T_h} \norm{h^2(\Box u - \chi \phi)}_{L^2(K)}^2
+ \sum_{F \in \mathcal F_h} h \norm{\jump{h\p_\nu u}}_{L^2(F)}^2
= 0,
    \end{align*}
and also $S(\phi) = 0$ and $E(U|_{t=0} - U_0) = E(U|_{t=T}) = 0$.

The equation $d\mathcal L(u, \phi) = 0$ can be written as 
    \begin{align}\label{fem}
A[(u, \phi), (v, \psi)] = h^{-\kappa} e(U_0, V|_{t=0}) + L(\psi) \quad 
\text{for all $(v,\psi) \in V_h^p \times V_h^q$},
    \end{align}
where the bilinear form $A$ is given by
    \begin{align*}
A[(u, \phi), (v, \psi)] 
&= 
h^{-\kappa}s(u,v) - h^\kappa s(\phi, \psi)
- h^2 c(\phi, \psi) 
+ h^{-\kappa} \sum_{\tau = 0,T} e(U|_{t=\tau}, V|_{t=\tau}) 
\\&\qquad+ a(v, \phi) + a(u, \psi)
\\&\qquad+ h^{4-\kappa} \tilde c(\phi, \psi) + h^{2-\kappa}\rho(v, \phi) + h^{2-\kappa}\rho(u, \psi).
    \end{align*}
Here $s$ is the bilinear form associated to the quadratic form $S$, and this lowercase--uppercase convention is systematically used also for other quadratic and bilinear forms in the paper.
Let us emphasize that all the bilinear forms $s$, $c$, $e$, $\tilde c$ and $\rho$ satisfy the same bound (\ref{bilinear_bound}) as $a$.


We define the residual norm by
    \begin{align}\label{def_tnorm}
\tnorm{(u,\phi)}^2 = h^{-\kappa} S(u) + h^\kappa S(\phi) + h^2 C(\phi) 
+ h^{-\kappa} \sum_{\tau = 0,T} E(U|_{t=\tau}).
    \end{align}

\begin{lemma}\label{lem_tnorm_is_norm}
Suppose that (A) holds. Then
$\tnorm{\cdot}$ is a norm on $V_h \times V_h$.
\end{lemma} 
\begin{proof}
Suppose $\tnorm{(u,\phi)} = 0$.
Then $\Box u = 0$ elementwise and $\jump{\p_\nu u} = 0$
for all internal faces. It follows that $\Box u = 0$ in the weak sense. 
As $u|_\Gamma = 0$ and $E(U|_{t=0}) = 0$, it follows that $u=0$. 
Similarly $\Box \phi = 0$ in the weak sense. As $\phi|_\Gamma = 0$ and $\widetilde C(\phi) = 0$, the distributed observability estimate, see Theorem \ref{th_obs_dist}, implies that $\phi = 0$. 
\end{proof}

\begin{lemma}\label{lem_stability}
For all sufficiently small $h$ and all $u,\phi \in H^2(M) + V_h$ there holds
    \begin{align*}
\tnorm{(u,\phi)}^2 \lesssim A[(u,\phi),(u,-\phi)].
    \end{align*}
\end{lemma} 
\begin{proof}
By the definition of $A$, we have
    \begin{align*}
A[(u, \phi),(u, -\phi)] 
=& \tnorm{(u,\phi)}^2 - h^{4-\kappa} \widetilde C(\phi).
    \end{align*}
As $\kappa < 2$ and $\chi \le 1$, $h^{4-\kappa} \widetilde C(\phi)$ can be absorbed by $h^2 C(\phi)$ for small $h > 0$.
\end{proof}

The previous two lemmas imply that (\ref{fem}) has a unique solution. Indeed, (\ref{fem}) is a square system of linear equations and the lemmas imply that $(u,\phi) = 0$ is the only solution when the right-hand side is zero. The right-hand side being zero is equivalent with $U_0 = 0$. 

\subsection{Error estimates}
Equation (\ref{fem}) defines a finite element method that is 
consistent in the sense that if
smooth enough $u$ and $\phi$ satisfy (\ref{control_eqs}), then (\ref{fem}) holds for $(u,\phi)$. 
This follows from the weak formulation (\ref{weak_form}) of (\ref{control_eqs}) together with the regularization vanishing for $(u,\phi)$.
In particular, if $(u_h, \phi_h) \in V_h^p \times V_h^q$ solves (\ref{fem})
then the following Galerkin orthogonality holds
\begin{equation}\label{eq:gal_ortho}
A[(u - u_h,\phi - \phi_h),(v,\psi)] = 0 \quad \mbox{ for all } (v,\psi)
\in V_h^p \times V_h^q.
\end{equation}

It is straightforward to see that for all $u,\phi,v,\psi \in H^2(M) + V_h$ there holds
    \begin{align}\label{A_cont}
A[(u, \phi), (v, \psi)] 
-( a(v, \phi) + a(u, \psi) ) 
\lesssim 
\tnorm{(u,\phi)} \tnorm{(v,\psi)}.
    \end{align}
We will need the following continuity estimates for $a$.

\begin{lemma}\label{lem_a_cont}
For all $u,\phi,v,\psi \in H^2(M) + V_h$ vanishing on $\Gamma$ there holds
    \begin{align*}
a(v, \phi) & \lesssim S^\frac12(\phi) \norm{v}_{H^1(\mathcal T_h)},
\\
a(u, \psi) & \lesssim \left(S^\frac12(u) + \sum_{\tau = 0,T} E^\frac12(U|_{t=\tau})\right) \norm{\psi}_{H^2(\mathcal T_h)}.
    \end{align*} 
\end{lemma}  
\begin{proof}
Recalling (\ref{int_by_parts}) we see that 
    \begin{align*}
a(v, \phi)
&= 
\int_{M} g(hdv, hd\phi) \d x 
-h(v, h\p_\nu \phi)_{L^2(\p M \setminus \Gamma)}
\\&=
\sum_{K \in \mathcal T_h}
\int_K v h^2\Box \phi \d x 
+
\sum_{F \in \mathcal F_h}
h \int_F v \jump{h\p_\nu \phi} \d s
    \end{align*}
and the first claimed estimate follows from the Cauchy--Schwarz inequality and the trace inequality (\ref{trace_cont}).

Let us now turn to the second estimate.
We have
    \begin{align*}
a(u, \psi)
&= 
\int_{M} g(hdu, hd\psi) \d x 
-h(u, h\p_\nu \psi)_{L^2(\p M \setminus \Gamma)}
\\&= 
\sum_{K \in \mathcal T_h}
\int_K h^2\Box u \psi \d x 
+
\sum_{F \in \mathcal F_h}
h\int_F \jump{h\p_\nu u} \psi \d s
\\&\qquad
+h(h\p_\nu u, \psi)_{L^2(\p M \setminus \Gamma)}
-h(u, h\p_\nu \psi)_{L^2(\p M \setminus \Gamma)},
    \end{align*}
and the second estimate follows.
\end{proof} 

Let us recall estimates for the Scott--Zhang interpolant
$i_h^p$ taking functions in $H^1(M)$, that vanish on $\Gamma$,
to $V_h^p$, see \cite{SZ90}.
For all $p \in \Np$ and $k =1, \dots, p+1$ there is $C > 0$ such that
for all $h>0$ and $u \in H^k(M)$
    \begin{align}\label{interp}
\norm{u- i_h^p u}_{H^k(\mathcal T_h)}
\le C \norm{(hD)^k u}_{L^2(M)}.
    \end{align}


\begin{theorem}\label{th:res_conv}
Suppose that (A) holds.
Let $\kappa < 2$, $p,q \in \Np$ and let $(u_h, \phi_h)$ in $V^p_h \times V^q_h$ be the solution of \eqref{fem}. Let $u \in H^{p+1}(M)$ and $\phi \in H^{q+1}(M)$ solve \eqref{control_eqs}. Then
\[
\tnorm{(u-u_h,\phi-\phi_h)} 
\lesssim 
h^{p+1-\frac \kappa 2} \|u\|_{H^{p+1}(M)}
+
h^{q+1+\frac \kappa 2} \|\phi\|_{H^{q+1}(M)}.
\]
In particular,
    \begin{align}\label{est_control}
\norm{\chi (\phi-\phi_h)}_{L^2(M)} 
\lesssim  
h^{p-\frac \kappa 2} \|u\|_{H^{p+1}(M)}
+
h^{q+\frac \kappa 2} \|\phi\|_{H^{q+1}(M)}.
    \end{align}
\end{theorem}
\begin{proof}
We write
    \begin{align}\label{errors}
w = u_h - u, 
\quad 
\eta = \phi_h - \phi, 
\quad
w_h = u_h - i_h^p u,
\quad 
\eta_h = \phi_h - i_h^q \phi.
    \end{align}
By Lemma \ref{lem_stability} and the Galerkin orthogonality (\ref{eq:gal_ortho}),
    \begin{align*}
\tnorm{(w,\eta)}^2 \lesssim
A[(w,\eta),(w,-\eta)]
= A[(w,\eta),(w-w_h,\eta_h-\eta)].
    \end{align*}
We write 
    \begin{align*}
w_i = i_h^p u - u, \quad
\eta_i = i_h^q \phi - \phi.
    \end{align*}
Observing that $w-w_h = w_i$ and $\eta_h - \eta = -\eta_i$,
it follows from (\ref{A_cont}) and Lemma \ref{lem_a_cont}
that
    \begin{align*}
\tnorm{(w,\eta)} &\lesssim
\tnorm{(w_i, \eta_i)} 
+ h^{-\frac \kappa 2}\norm{w_i}_{H^1(\mathcal T_h)}
+ h^{\frac \kappa 2}\norm{\eta_i}_{H^2(\mathcal T_h)}.
    \end{align*}
Recalling the scaling in (\ref{def_tnorm}) and using the bound (\ref{bilinear_bound}), with $a$ replaced by $s$, $c$ and $e$, we see that
    \begin{align*}
\tnorm{(w_i, \eta_i)}
\lesssim 
h^{-\frac \kappa 2}\norm{w_i}_{H^2(\mathcal T_h)}
+ h^{\frac \kappa 2}\norm{\eta_i}_{H^2(\mathcal T_h)}.
    \end{align*}
Finally, using (\ref{interp}),
    \begin{align*}
h^{-\frac \kappa 2}\norm{w_i}_{H^2(\mathcal T_h)}
+ h^{\frac \kappa 2}\norm{\eta_i}_{H^2(\mathcal T_h)}
\lesssim 
h^{p+1-\frac \kappa 2} \|u\|_{H^{p+1}(M)}
+
h^{q+1+\frac \kappa 2} \|\phi\|_{H^{q+1}(M)}.
    \end{align*}
\end{proof}

Recall that if $(u_0,u_1) \in H^{k+1}(\Omega)\times H^{k}(\Omega)$
satisfies the compatibility conditions of order $k$,
then the unique solution $(u,\phi)$ to (\ref{control_eqs}) is in $H^{k+1}(M) \times H^k(M)$.
Hence we can take $p \le k$ and $q \le k-1$. Choosing $\kappa = 0$ and $p=q \le k-1$ leads to the convergence rate (\ref{conv_intro}) stated in the introduction.

Under the assumptions of Theorem \ref{th:res_conv}, it is possible to show that 
    \begin{align*}
\norm{u - u_h}_{L^\infty(0,T; L^2(\Omega))}
+ 
\norm{\p_t(u - u_h)}_{L^2(0,T; H^{-1}(\Omega))}
\lesssim  
h^p \|u\|_{H^{p+1}(M)}+h^q\|\phi\|_{H^{q+1}(M)}. \end{align*}
We do not detail the proof here, but refer to \cite[Theorem 4.4]{BFMO_2021} for a
similar analysis.
The weak norms reflect the fact that the forward problem does not
enjoy the classical energy stability of the wave equation. Instead
error estimates are derived using continuum estimates on a level
dictated by the regularity of $\Box(u-u_h)$. This quantity is in
$H^{-1}(M)$, and not likely in a better space, resulting in the above estimate. 
Continuum theory at this energy level is reviewed in an appendix
below, see Remark \ref{rem_opt_rough} in particular. 
\begin{remark}
Observe that the corresponding stability estimates for unique
continuation given in \cite[Theorem
2.2]{BFL20}, \cite[Theorem 1.1]{BFMO_2021} are inaccurate, claiming
control of $\|\partial_t u\|_{L^\infty(0,T;H^{-1}(\Omega))}$ when the
best quantity that can be controlled (as shown in appendix below,
Theorem \ref{th_obs_dist}, Remark \ref{rem_obs_variants} and
Proposition \ref{prop_energy}) is
$\|\partial_t u\|_{L^2(0,T;H^{-1}(\Omega))}+\|\partial_t
u\vert_{t=0}\|_{H^{-1}(\Omega)}+\|\partial_t
u\vert_{t=T}\|_{H^{-1}(\Omega)}$. The results in the above references
are nevertheless correct without further modifications after
correction of the stability norm for the error analysis.
\end{remark}

However, we obtain a better approximation simply by solving
    \begin{align}\label{wave_direct}
\begin{cases}
\Box u = f, 
\\
u|_{x \in \p \Omega}=0,
\\
u|_{t=T} = 0,\ \p_t u|_{t=T} = 0,
\end{cases}
    \end{align}
with $f = \chi \phi_h$.
We will detail the arguments in an abstract setting
below.

Let $\Box_h$ denote a stable discrete wave operator with vanishing initial and boundary conditions such that the
following standard stability estimate holds for the solution $u_h$ to
$\Box_h u_h = f$,
\[
\tnorm{u_h}_E := 
\norm{u_h}_{L^\infty(0,T; H^1(\Omega))}
+ \norm{\p_t u_h}_{L^\infty(0,T; L^2(\Omega))}
\lesssim \|f\|_{L^2(M)}.
\]
We also assume that the following optimal error estimate holds: if  $u$ is the
solution to (\ref{wave_direct}), then there holds
\[
\tnorm{ u - u_h}_E  \lesssim h^p.
\]
For a high order scheme satisfying these assumptions see for instance
\cite{FP96}.
Let now $u$ be the solution to (\ref{wave_direct}) with $f=\chi\phi$, $v_h$ the
solution to $\Box_h v_h = \chi \phi$ and $u_h$ the solution to $\Box u_h =
\chi \phi_h$. It then follows by the above inequalities that 
\begin{align*}
\tnorm{ u - u_h}_E \leq \tnorm{ u - v_h}_E + \tnorm{ v_h - u_h}_E
\lesssim h^p + \|\chi (\phi - \phi_h)\|_{L^2(M)} \lesssim h^p + h^q.
\end{align*}
Here we used the properties of the method $\Box_h$ and Theorem
\ref{th:res_conv}.

\section{Distributed control with limited regularity}

In this section we will study the finite element method (\ref{fem})
in the case that the continuum solution $(u,\phi)$ to the control problem (\ref{control_eqs}) is in the natural energy class $H^1(M) \times L^2(M)$. We make the standing assumption that (A) holds, so that \eqref{fem} has a unique solution.  

\begin{lemma}\label{lem_rough_box}
Let $(u_h, \phi_h) \in V_h^p \times V_h^q$ be the solution 
of \eqref{fem} with $\kappa = 0$.
Then
    \begin{align*}
h^{-1}\norm{h^2\Box \phi_h}_{H^{-1}(M)}
&\lesssim 
h^\frac12\norm{hu_1}_{L^2(\Omega)} + 
\tnorm{(u_h, \phi_h)},
\\
h^{-1}\norm{h^2\Box u_h}_{H^{-1}(M)}
&\lesssim 
h^\frac12\norm{u_0}_{L^2(\Omega)} + 
\tnorm{(u_h, \phi_h)}.
    \end{align*}
 
\end{lemma}

\begin{proof}
To establish the first claimed inequality, we will show for $v \in H_0^1(M)$
that
    \begin{align*}
\int_M g(hdv,hd\phi_h) \d x
\lesssim
(h^\frac12 \norm{hu_1}_{L^2(\Omega)} + \tnorm{(u_h, \phi_h)}) \norm{h \nabla v}_{L^2(M)}.
    \end{align*}
We have
    \begin{align*}
\int_M g(hdv,hd\phi_h) \d x
= 
\sum_{K \in \mathcal T_h}
\int_K v h^2\Box \phi_h \d x 
+
\sum_{F \in \mathcal F_h}
h \int_F v \jump{h\p_\nu \phi_h} \d s.
    \end{align*}
Let $v_h \in V_h^p$ be the Scott--Zhang interpolant of $v$,
and apply the above equation with $v$ replaced by $v - v_h$.
Then
    \begin{align*}
\int_M g(hd(v-v_h),hd\phi_h) \d x
\lesssim
\tnorm{(0, \phi_h)} \norm{h \nabla v}_{L^2(M)}.
    \end{align*}
Moreover, using (\ref{fem})
    \begin{align*}
&-\int_M g(hdv_h,hd\phi_h) \d x
= 
h^{-\kappa} s(u_h,v_h) 
+
h^{-\kappa} \sum_{\tau = 0,T} h (h \p_t u_h|_{t=\tau}, h\p_t v_h|_{t=\tau})_{L^2(\Omega)} 
\\&\qquad+ 
h^{2-\kappa}\rho(v_h, \phi_h)
- h^{-\kappa} h (hu_1, h\p_t v_h|_{t=0})_{L^2(\Omega)}.
    \end{align*}
Hence, using $\kappa = 0$,
    \begin{align*}
\int_M g(hdv_h,hd\phi_h) \d x
\lesssim 
(h^\frac12 \norm{hu_1}_{L^2(\Omega)} + \tnorm{(u_h, \phi_h)})
\norm{h \nabla v_h}_{L^2(M)}.
    \end{align*}
 
Let us now show for $\psi \in H_0^1(M)$
    \begin{align*}
\int_M g(hdu_h,hd\psi) \d x
\lesssim
(h^\frac12\norm{u_0}_{L^2(\Omega)} + 
\tnorm{(u_h, \phi_h)}) h\norm{\psi}_{H^1(M)}.
    \end{align*}
Let $\psi_h \in V_h^p$ be the Scott--Zhang interpolant of $\psi$.
Analogously with the above, we have
    \begin{align*}
\int_M g(hdu_h,hd(\psi-\psi_h)) \d x
\lesssim
\tnorm{(u_h, 0)} \norm{h \nabla \psi}_{L^2(M)}.
    \end{align*}
Moreover,
    \begin{align*}
&-\int_M g(hdu_h,hd\psi_h) \d x
=
h(u_h, h\p_\nu \psi_h)_{L^2(\p M \setminus \Gamma)}
- h^\kappa s(\phi_h, \psi_h)
- h^2 c(\phi_h, \psi_h) 
\\&\qquad
+ h^{4-\kappa} \tilde c(\phi_h, \psi_h) 
+ h^{2-\kappa}\rho(u_h, \psi_h) 
+ h(u_0, h\p_t \psi_h|_{t=0})_{L^2(\Omega)},
    \end{align*}
and the second claimed inequality follows. 
\end{proof} 

\begin{lemma}\label{lem_rough_tnorm}
Let $(u_h, \phi_h)$ be the solution of \eqref{fem} with $\kappa = 0$.
Then 
    \begin{align}\label{rough_tnorm}
\tnorm{(u_h, \phi_h)} \lesssim h^\frac12 \norm{u_0}_{L^2(\Omega)} + h \norm{u_0}_{H^1_0(\Omega)} + h \norm{u_1}_{L^2(\Omega)}.
    \end{align}
\end{lemma}
\begin{proof}
By Lemma \ref{lem_stability} there holds
    \begin{align*}
\tnorm{(u_h,\phi_h)}^2 
&\lesssim 
A[(u_h,\phi_h),(u_h,-\phi_h)]
= 
e(U_0, U_h|_{t=0}) - L(\phi_h)
\\&\le 
h^\frac12 (\norm{u_0}_{L^2(\Omega)} + \norm{hu_1}_{L^2(\Omega)})
E^\frac12(U_h|_{t=0})
\\&\qquad+ 
h^2 (\norm{u_0}_{H^1_0(\Omega)} + \norm{u_1}_{L^2(\Omega)})
(\norm{\phi_h|_{t=0}}_{L^2(\Omega)} + \norm{\p_t \phi_h|_{t=0}}_{H^{-1}(\Omega)}).
    \end{align*}
By the distributed observability estimate, see Theorem \ref{th_obs_dist} in Appendix \ref{appendix_cont}, 
    \begin{align*}
&\norm{\phi_h|_{t=0}}_{L^2(\Omega)}
+
\norm{\p_t \phi_h|_{t=0}}_{H^{-1}(\Omega)}
\lesssim
C^\frac12(\phi_h) +
\norm{\Box \phi_h}_{H^{-1}(M)}.
    \end{align*}
Recalling that $h C^\frac12(\phi_h) \lesssim \tnorm{(0,\phi_h)}$,
and using Lemma \ref{lem_rough_box}, we obtain
    \begin{align*}
& h(\norm{\phi|_{t=0}}_{L^2(\Omega)}
+
\norm{\p_t \phi|_{t=0}}_{H^{-1}(\Omega)})
\lesssim
h^\frac12\norm{hu_1}_{L^2(\Omega)} + 
\tnorm{(u_h, \phi_h)}.
    \end{align*}
As also $E^\frac12(U_h|_{t=0}) \le \tnorm{(u_h, \phi_h)}
$, we have 
    \begin{align*}
&\tnorm{(u_h,\phi_h)}^2 
\\&\quad\lesssim 
(h^\frac12 \norm{u_0}_{L^2(\Omega)} + h \norm{u_0}_{H^1_0(\Omega)} + h \norm{u_1}_{L^2(\Omega)})(h^\frac12\norm{hu_1}_{L^2(\Omega)} + 
\tnorm{(u_h, \phi_h)}),
    \end{align*}
leading to (\ref{rough_tnorm}).   
\end{proof} 

\begin{lemma}\label{lem_rough_bounded}
Let $(u_h, \phi_h)$ be the solution of \eqref{fem} with $\kappa = 0$
and $u_0 = 0$. Then
    \begin{align*}
\norm{u_h}_{L^2(M)} + \norm{\phi_h}_{L^2(M)} \lesssim \norm{u_1}_{L^2(\Omega)}.
    \end{align*}
\end{lemma}
\begin{proof}
Lemma \ref{lem_rough_box} implies 
    \begin{align*}
\norm{\Box u_h}_{H^{-1}(M)}
+ \norm{\Box \phi_h}_{H^{-1}(M)}
\lesssim \norm{u_1}_{L^2(\Omega)}.
    \end{align*}
Moreover, it follows from (\ref{rough_tnorm}) that 
    \begin{align*}
C^\frac12(\phi_h) \lesssim \norm{u_1}_{L^2(\Omega)},
\quad
\norm{\p_t^j u_h|_{t=0}}_{L^2(\Omega)} \lesssim h^{\frac12-j} \norm{u_1}_{L^2(\Omega)},
\qquad j=0,1.
    \end{align*}
The bound 
$\norm{\phi_h}_{L^2(M)} \le C \norm{u_1}$
follows from the distributed observability estimate, 
see Remark \ref{rem_obs_variants} below. 
It remains to show the same bound for $u_h$.
We face the complication that the above estimates do not allow us to conclude that $\p_t u_h|_{t=0}$ is bounded. 

To overcome this, we will employ $\tilde u_h \in V_h^p$
that coincides with $u_h$ on $\p M$ and satisfies (\ref{eq:negortho}) and (\ref{eq:discapprox}) below.
We have 
    \begin{align*}
\norm{\Box \tilde u}_{H^{-1}(M)} \lesssim \norm{u_1}_{L^2(\Omega)}.
    \end{align*}
Indeed, for any $v \in H^1_0(M)$ there holds, using (\ref{inverse_disc}) and (\ref{eq:discapprox}),
    \begin{align*}
&(h^2\Box \tilde u_h, v)_{L^2(M)}
= 
\int_M g(hd \tilde u_h, hd v) \d x
\\&\quad\le 
\int_M g(hd u_h, hd v) \d x
+ \norm{h\nabla(\tilde u_h - u_h)}_{L^2(M)} \norm{h \nabla v}_{L^2(M)}
\\&\quad\lesssim
\norm{h^2\Box u_h}_{H^{-1}(M)} \norm{v}_{H^1_0(M)}
+ h^{\frac12}\|h(\partial_t u_h|_{t=0} - u_1)\|_{L^2(\Omega)}  \norm{h \nabla v}_{L^2(M)}.
\\&\quad\lesssim
h^2(\norm{u_1}_{L^2(\Omega)} + h^{\frac12}\|\partial_t u_h|_{t=0}\|_{L^2(\Omega)}) \le h^2 \norm{u_1}_{L^2(\Omega)}.
    \end{align*}
Moreover, using (\ref{eq:negortho}),
    \begin{align*}
&\|\partial_t \tilde u_h|_{t=0}\|_{H^{-1}(\Omega)} 
\le 
\|\partial_t \tilde u_h|_{t=0} - u_1\|_{H^{-1}(\Omega)} 
+ \|u_1\|_{H^{-1}(\Omega)} 
\\&\quad\lesssim  h\|\partial_t u_h|_{t=0} - u_1\|_{L^2(\Omega)} 
+ \|u_1\|_{H^{-1}(\Omega)} 
\le h^\frac12 \norm{u_1}_{L^2(\Omega)} + \|u_1\|_{H^{-1}(\Omega)}.
    \end{align*}
Recalling that $\tilde u_h$ coincides with $u_h$ on $\p M$, we conclude that
    \begin{align*}
\norm{\tilde u_h}_{L^2(M)} \lesssim \norm{u_1}_{L^2(\Omega)}
    \end{align*} 
follows from an energy estimate, see Proposition \ref{prop_energy}
in Appendix \ref{appendix_cont}. 
Finally, using
(\ref{eq:discapprox}),
    \begin{align*}
\norm{u_h}_{L^2(M)}
&\le 
\|u_h - \tilde u_h\|_{L^2(M)} 
+ \norm{\tilde u_h}_{L^2(M)}
\\&\lesssim  
h^{\frac12}\|h(\partial_t u_h|_{t=0} - u_1)\|_{L^2(\Omega)} + \norm{u_1}_{L^2(\Omega)}
\lesssim \|u_1\|_{L^2(\Omega)}.
    \end{align*}
\end{proof} 

\begin{lemma}
Let $p \in \Np$ and consider a family $u_h \in V_h^p$, $h > 0$.
Let $u_1 \in L^2(\Omega)$.
Then there is a family $\tilde u_h \in V_h^p$, $h > 0$, such that 
$\tilde u_h|_{\partial M} = u_h|_{\partial M}$ and
\begin{align}\label{eq:negortho}
h^{-1}\|\partial_t \tilde u_h|_{t=0} - u_1\|_{H^{-1}(\Omega)} 
&\lesssim  \|\partial_t u_h|_{t=0} - u_1\|_{L^2(\Omega)},
\\\label{eq:discapprox}
\|u_h - \tilde u_h\|_{L^2(M)} 
&\lesssim  
h^{\frac12}\|h(\partial_t u_h|_{t=0} - u_1)\|_{L^2(\Omega)}.
\end{align}
\end{lemma}
\begin{proof}
Let us consider the trace mesh at $t=0$,
    \begin{align*}
\mathcal F_{h,0} = \{\p K \cap \{t=0\} : K \in \mathcal T_h \}.
    \end{align*}
We decompose $\mathcal F_{h,0}$ into a set
of $N_h$ disjoint patches $\p \mathcal{P}_i$, $i=1,\dots,N_h$, such that each patch contains several element faces but their area
and diameter satisfy
    \begin{align*}
h^n \lesssim |\p \mathcal P_i| \lesssim h^n,
\quad h \lesssim \mathrm{diam}(P_i) \lesssim h.
    \end{align*}
Then we define disjoint patches $\mathcal{P}_i$
consisting of elements of $\mathcal{T}_h$
so that 
    \begin{align*}
\p \mathcal{P}_i = \mathcal{P}_i \cap \{t=0\}
    \end{align*}
and that $h^{n+1} \lesssim |\mathcal P_i| \lesssim h^{n+1}$. 
Now we define the functions $p_i \in V_h^1$
such that $\supp(p_i) \subset \mathcal{P}_i$ and $p_i(x) = 1$
for every node $x$ in the interior of $\mathcal{P}_i$. 
We require that the patches $\p \mathcal P_i$
are large enough so that, writing
    \begin{align*}
\alpha_i = \int_{\p \mathcal P_i} \p_t p_i|_{t=0} ~\mbox{d}s,
\quad 
\beta_i = \norm{\p_t p_i|_{t=0}}_{L^2(\p \mathcal P_i)},
\quad
\gamma_i = \norm{p_i}_{L^2(\mathcal P_i)},
    \end{align*}
there holds $h^{n-1} \lesssim \alpha_i \lesssim h^{n-1}$, $h^{\frac n 2 -1} \lesssim \beta_i \lesssim h^{\frac n 2 -1}$
and $h^{\frac 1 2(n+1)} \lesssim \gamma_i \lesssim h^{\frac 1 2(n+1)}$.

We set 
    \begin{align*}
\tilde u_h = u_h + \sum_{i = 1}^{N_h} w_i p_i,
\quad 
w_i = - \alpha_i^{-1} \int_{\partial \mathcal{P}_i} (\partial_t u_h|_{t=0} - u_1) ~\mbox{d}s.
    \end{align*}
Then 
    \begin{align}\label{tildeu_ortho}
\int_{\partial \mathcal{P}_i} (\partial_t \tilde u_h|_{t=0} - u_1) ~\mbox{d}s
= 0.
    \end{align}
To establish (\ref{eq:negortho}) we let $v \in H_0^1(\Omega)$ and  show that 
    \begin{align*}
(\partial_t \tilde u_h(\cdot,0) - u_1, v)_{L^2(\Omega)}
\lesssim 
\norm{\partial_t u_h(\cdot,0) - u_1}_{L^2(\Omega)} 
\norm{h \nabla v}_{L^2(\Omega)}.
    \end{align*}
Let $\bar v \in L^2(\Omega)$ be equal to the average of $v$ on
each patch $\p P_i$, that is, 
    \begin{align*}
\bar v \vert_{\partial \mathcal{P}_i} = |\partial \mathcal{P}_i|^{-1}
\int_{\partial \mathcal{P}_i} v ~\mbox{d}s.
    \end{align*}
Now (\ref{tildeu_ortho}) implies
$(\partial_t \tilde u_h|_{t=0} - u_1, \bar v)_{L^2(\Omega)} = 0$, and
    \begin{align*}
&(\partial_t \tilde u_h|_{t=0} - u_1, v)_{L^2(\Omega)}
=
(\partial_t \tilde u_h|_{t=0} - u_1, v - \bar v)_{L^2(\Omega)}
\\&\quad\le 
\norm{\partial_t \tilde u_h|_{t=0} - u_1}_{L^2(\Omega)} 
\norm{v - \bar v}_{L^2(\Omega)}
\lesssim 
\norm{\partial_t \tilde u_h|_{t=0} - u_1}_{L^2(\Omega)} 
\norm{h\nabla v}_{L^2(\Omega)}.
    \end{align*}
Here we used the Poincar\'e inequality as stated for example in \cite{Ern2017}.
To establish (\ref{eq:negortho}) it remains to show that 
    \begin{align*}
\norm{\partial_t \tilde u_h|_{t=0} - u_1}_{L^2(\Omega)} 
\lesssim 
\norm{\partial_t u_h|_{t=0} - u_1}_{L^2(\Omega)}.
    \end{align*}
Using the fact that the patches $\p \mathcal P_i$ are disjoint, we have
    \begin{align*}
\norm{\partial_t \tilde u_h|_{t=0} - u_1}_{L^2(\Omega)}
\le \norm{\partial_t u_h|_{t=0} - u_1}_{L^2(\Omega)}
+ \sum_{i=1}^{N_h} |w_i| \norm{\p_t p_i|_{t=0}}_{L^2(\p \mathcal P_i)}.
    \end{align*}
Recalling that $\alpha_i$ behaves like $h^{n-1}$ and $\beta_i$ like $h^{\frac n 2 -1}$, we obtain using the Cauchy--Schwarz inequality
    \begin{align*}
|w_i| \norm{\p_t p_i|_{t=0}}_{L^2(\p \mathcal P_i)} 
&=
\alpha_i^{-1} \beta_i \left| \int_{\partial \mathcal{P}_i} (\partial_t u_h|_{t=0} - u_1) ~\mbox{d}s \right|
\\& \lesssim 
h^{1-n} h^{\frac n 2 -1} h^{\frac{n}{2}} 
\|\partial_t u_h|_{t=0} - u_1\|_{L^2(\partial \mathcal{P}_i)}
= 
\|\partial_t u_h|_{t=0} - u_1\|_{L^2(\partial \mathcal{P}_i)}.
    \end{align*}

Let us now turn to \eqref{eq:discapprox}. Note that
\begin{align*}
\|u_h - \tilde u_h \|^2_{L^2(M)} =\sum_{i=1}^{N_h} |w_i|^2 \|p_i\|^2_{L^2(\mathcal{P}_i)}.
\end{align*}
Recalling that $\gamma_i^2$ behaves like $h^{n+1}$, we obtain
    \begin{align*}
|w_i|^2 \|p_i\|^2_{L^2(\mathcal{P}_i)}
= |w_i|^2 \gamma_i^2 
\lesssim h^{n+1} h^{2(1-n)} h^{n} 
\|\partial_t u_h|_{t=0} - u_1\|_{L^2(\partial \mathcal{P}_i)}^2,
    \end{align*}
leading to 
    \begin{align*}
\|u_h - \tilde u_h \|^2_{L^2(M)} \lesssim h^3 \|\partial_t u_h|_{t=0} - u_1\|_{L^2(\Omega)}^2.
    \end{align*}
\end{proof}

\begin{theorem}\label{th_rough}
Suppose that $u_0 = 0$ and $u_1 \in L^2(\Omega)$.
Let $(u_h, \phi_h)$ be the solution of \eqref{fem} with $\kappa = 0$,
and let $(u,\phi)$ be the solution of \eqref{control_eqs}. Then
there is a sequence $h_j \to 0$ such that $(u_{h_j}, \phi_{h_j})$
converges weakly to $(u,\phi)$ in $L^2(M)$.
\end{theorem} 




\begin{proof}
By Lemma \ref{lem_rough_bounded} both $u_h$ and $\phi_h$ are bounded in $L^2(M)$.
Thus there is a sequence $h_j \to 0$ such that 
$(u_{h_j}, \phi_{h_j})$
converges weakly to a function $(u_*,\phi_*)$ in $L^2(M)$.
By Lemma \ref{lem_uniq} it is enough to show that $(u_*,\phi_*)$ satisfies \eqref{control_eqs}.

As the embedding $H^{-\epsilon}(M) \subset L^2(M)$ is compact for $\epsilon > 0$, by passing to a subsequence, we may assume that $(u_{h_j}, \phi_{h_j}) \to (u_*,\phi_*)$ in $H^{-\epsilon}(M)$.
By Lemmas \ref{lem_rough_box} and \ref{lem_rough_tnorm}
we may further assume that 
$(\Box u_{h_j}, \Box \phi_{h_j}) \to (\Box u_*,\Box \phi_*)$ in $H^{-\epsilon-1}(M)$. For $\epsilon < 1/2$ it follows from Lemma \ref{lem_parthypo} that 
    \begin{align*}
0 = (u_{h_j}|_\Gamma, \phi_{h_j}|_\Gamma) \to (u_*|_\Gamma,\phi_*|_\Gamma).
    \end{align*}
Thus $(u_*,\phi_*)$ satisfies the homogeneous lateral boundary conditions in \eqref{control_eqs}.

For any $\psi \in C^\infty(M)$ with $\psi|_\Gamma = 0$ and any $v \in C_0^\infty(M)$ there holds
    \begin{align}\label{rough_eq_conv}
h^{-2}(a(u_h, \psi) - h^2 c(\phi_h, \psi) - L(\psi))
&\to 0, 
\\\label{rough_eq_conv2}
h^{-2} a(v, \phi_h) &\to 0,
    \end{align}
as $h \to 0$. 
Before showing (\ref{rough_eq_conv})--(\ref{rough_eq_conv2}),
let us show that they imply that $(u_*, \phi_*)$ satisfies \eqref{control_eqs}. The equation $\Box \phi_* = 0$ follows immediately from (\ref{rough_eq_conv2}).
Observe that 
    \begin{align*}
h^{-2} a(u_h, \psi) = (u_h, \Box \psi)_{L^2(M)}
\to (u_*, \Box \psi)_{L^2(M)}.
    \end{align*}
It follows from (\ref{rough_eq_conv}) that for any $\psi \in C^\infty(M)$ vanishing on $\Gamma$ there holds
    \begin{align}\label{ustar_eq}
(u_*, \Box \psi)_{L^2(M)} = c(\phi_*, \psi) + (u_1, \psi|_{t=0})_{L^2(\Omega)} - (u_0, \p_t \psi|_{t=0})_{L^2(\Omega)}.
    \end{align}
In particular, taking $\psi \in C_0^\infty(M)$
we see that $\Box u_* = \chi \phi_*$.

To show that $(u_*, \phi_*)$ satisfies \eqref{control_eqs},
it remains to verify the initial and final conditions for $u_*$. 
We have $u_* \in L^2(M)$ and 
    \begin{align*}
\p_t^2 u_* = \Delta u_* + \chi \phi_* \in L^2(0,T;H^{-2}(\Omega)).
    \end{align*}
Now \cite[Theorem 3.1, p. 19]{Lions1972}
gives 
    \begin{align*}
u_* \in C(0,T;H^{-\frac12}(\Omega)),
\quad
\p_t u_* \in C(0,T;H^{-\frac32}(\Omega)).
    \end{align*}
Taking $\psi(t,x) = \psi_0(t) \psi_1(x)$, with 
$\psi_0 \in C^\infty(0,T)$ and $\psi_1 \in C_0^\infty(\Omega)$,
we integrate by parts
    \begin{align*}
(u_*, \Box \psi)_{L^2(M)} 
&= 
\int_0^T \pair{u_*, \psi_1} \p_t^2 \psi_0 \d t
- \int_0^T \pair{u_*, \Delta \psi_1} \psi_0 \d t
\\&= 
\int_0^T (\chi \phi_*, \psi_1)_{L^2(\Omega)} \psi_0 \d t
+ \left[\pair{u_*, \psi_1} \p_t \psi_0 - \pair{\p_t u_*, \psi_1} \psi_0\right]_{t=0}^{t=T},
    \end{align*}
where $\pair{\cdot, \cdot}$ is the pairing between distribution and test functions on $\Omega$. Comparison with (\ref{ustar_eq}) shows that 
$u_*$ satisfies the initial and final conditions in \eqref{control_eqs}.

Let us now show (\ref{rough_eq_conv}). 
Denote by $\psi_h$ the Scott--Zhang interpolant of $\psi$.
By (\ref{fem})
    \begin{align*}
&a(u_h, \psi) - h^2 c(\phi_h, \psi) - L(\psi)
\\&\quad=
a(u_h, \phi - \psi_h) - h^2 c(\phi_h, \psi- \psi_h) - L(\psi- \psi_h)
\\&\qquad\quad
+ s(\phi_h, \psi_h) - h^4 \tilde c(\phi_h, \psi_h) - h^2 \rho(u_h, \psi_h).
    \end{align*}
Using the continuity of $a$ in Lemma \ref{lem_a_cont}, the interpolation estimate (\ref{interp}), and the bound (\ref{rough_tnorm}) for the residual norm, we obtain
    \begin{align*}
|a(u_h, \psi - \psi_h)| \lesssim \tnorm{(u_h, \phi_h)} \norm{(hD)^2 \psi}_{L^2(M)} \lesssim h \norm{u_1}_{L^2(M)} \norm{(hD)^2 \psi}_{L^2(M)}.
    \end{align*}
Recalling that $h C^\frac12(\phi_h) \lesssim \tnorm{(0,\phi_h)}$, 
we use the continuity (\ref{bilinear_bound}) for $c$ and the interpolation estimate (\ref{interp}) to get 
    \begin{align*}
h^2 |c(\phi_h, \psi- \psi_h)|
\le h^2 C^\frac12(\phi_h) C^\frac12(\psi- \psi_h)
\lesssim h \tnorm{(0,\phi_h)} \norm{(hD)^2 \psi}_{L^2(M)}.
    \end{align*}
Using once again (\ref{interp}),
    \begin{align*}
|L(\psi- \psi_h)| = h^2 |(u_1, (\psi - \psi_h)|_{t=0})_{L^2(\Omega)}|
\lesssim h^{3/2} \norm{u_1}_{L^2(\Omega)} \norm{(hD)^2 \psi}_{L^2(M)}.
    \end{align*}
Turning to the first term related to regularization, we have
    \begin{align*}
|s(\phi_h, \psi_h)|
\le S^\frac12(\phi_h) S^\frac12(\psi_h),
    \end{align*}
where the first factor is bounded by $\tnorm{(0,\phi_h)} \lesssim h \norm{u_1}_{L^2(\Omega)}$,
and the second satisfies 
    \begin{align*}
S(\psi_h) &\lesssim \sum_{K \in \mathcal T_h} (\norm{h^2\Box (\psi_h - \psi)}_{L^2(K)}^2 + \norm{h^2\Box \psi}_{L^2(K)}^2)
+ \sum_{F \in \mathcal F_h} h \norm{\jump{h\p_\nu (\psi_h - \psi)}}_{L^2(F)}^2
\\&\lesssim \norm{(hD)^2 \psi}_{L^2(M)}^2.
    \end{align*}
Finally,
    \begin{align*}
h^4 |\tilde c(\phi_h, \psi_h)|
&\lesssim h^3 \tnorm{(0,\phi_h)} \norm{\psi}_{L^2(M)},
\\
h^2 |\rho(u_h, \psi_h)|
&\lesssim h^2 \tnorm{(u_h, 0)} \norm{\psi}_{L^2(M)}
\lesssim h^3 \norm{u_1} \norm{\psi}_{L^2(M)},
    \end{align*}
and (\ref{rough_eq_conv}) follows. 
 
We turn to (\ref{rough_eq_conv2}).
Denote by $v_h$ the Scott--Zhang interpolant of $v$.
By (\ref{fem})
    \begin{align*}
a(v,\phi_h) = a(v-v_h, \phi_h) - s(u_h, v_h) + h^2 \rho(v_h, \phi_h).
    \end{align*}
Similarly to the bounds above, we have 
    \begin{align*}
|a(v-v_h, \phi_h)| + |s(u_h, v_h)|
&\lesssim 
h \norm{u_1}_{L^2(\Omega)} \norm{(hD)^2 v}_{L^2(M)},
\\
h^2 |\rho(v_h, \phi_h)| 
&\lesssim
h \tnorm{(0,\phi_h)} \norm{h^2 \Box v}_{L^2(M)},
    \end{align*}
and (\ref{rough_eq_conv2}) follows. This finishes the proof that  
$(u_*,\phi_*)$ satisfies \eqref{control_eqs}.
\end{proof} 

\section{Boundary control}

Let us begin by formulating our assumptions on the cutoff function $\chi$ in (\ref{wave_eq_intro_bd}).
We consider a function of the form
    \begin{align*}
\chi(t,x) = \chi_0(t) \chi_1^2(x),
    \end{align*}
where $\chi_0 \in C_0^\infty([0,T])$
and $\chi_1 \in C^\infty(\Gamma)$  take values in $[0,1]$,
and suppose that 
\begin{itemize}
\item[(A')] $\chi = 1$
on open $(a,b) \times \omega \subset \Gamma$
satisfying the geometric control condition.
\end{itemize}
In the case of boundary control, the geometric control condition
means that every compressed generalized bicharacteristic intersects the set $(a,b) \times \omega$, when projected to $M$.
Moreover, the intersection must happen at a nondiffractive point and the lightlike lines must have finite order of contact with $\Gamma$. We refer again to \cite{BLRII} for the definitions. 

We let $V \in C^\infty(\Omega)$ and consider the boundary control problem for the following operator
    \begin{align}\label{def_P}
P = \Box + V.
    \end{align}
Let $(u_0, u_1) \in L^2(\Omega) \times H^{-1}(\Omega)$. 
Then the distributed control problem for $P$ can be solved by finding $(u,\phi) \in L^2(M) \times H^1(M)$ such that 
    \begin{align}\label{control_eqs_bd}
&\begin{cases}
P u = 0,
\\
u|_\Gamma = \chi \p_\nu \phi,
\\
u|_{t=0} = u_0,\ \p_t u|_{t=0} = u_1,
\\
u|_{t=T} = 0,\ \p_t u|_{t=T} = 0,
\end{cases}
\quad \begin{cases}
P \phi = 0,
\\
\phi|_\Gamma =0.
\end{cases}
    \end{align}
If $(u_0,u_1) \in H^{k+1}(\Omega)\times H^{k}(\Omega)$
satisfies the compatibility conditions of order $k$,
then the unique solution $(u,\phi)$ to (\ref{control_eqs_bd}) satisfies 
    \begin{align}\label{phi_init}
\phi|_{t=T} \in H^{k+2}(\Omega), 
\quad 
\p_t \phi|_{t=T} \in H^{k+1}(\Omega), 
    \end{align}
and this initial data for $\phi$ satisfies the compatibility conditions of order $k+1$, see \cite[Thëorem 5.4]{Ervedoza2010}.
It follows that $\phi \in H^{k+2}(M)$ and $u \in H^{k+1}(M)$, and the convergence proof for our finite element method is again based on this regularity.

Although uniqueness of the solution $(u,\phi)$ to (\ref{control_eqs_bd}) is implictly contained in \cite{Ervedoza2010}, we give a short proof. This illustrates the difference in natural regularities between the distributed and boundary control cases. 

\begin{lemma}
Suppose that (A') holds. Let $(u, \phi) \in L^2(M) \times H^1(M)$
solve (\ref{control_eqs_bd}) with $u_0 = u_1 = 0$.
Then $u = \phi = 0$.
\end{lemma}
\begin{proof}
For the convenience of the reader, we show first that 
    \begin{align}\label{phi_reg_bd}
\phi \in C(0,T; H^1_0(\Omega)) \cap C^1(0,T; L^2(\Omega)),
\quad \p_\nu \phi|_\Gamma \in L^2(\Gamma).
    \end{align}
The proof of this fact is very similar to the proof of Lemma \ref{lem_linfty_ptu}. The standard energy estimate implies that for all $s \in (0,T)$,
    \begin{align*}
\norm{\phi|_{t=0}}_{H^1(\Omega)} 
+
\norm{\p_t\phi|_{t=0}}_{L^2(\Omega)} 
\lesssim
\norm{\phi|_{t=s}}_{H^1(\Omega)} 
+
\norm{\p_t\phi|_{t=s}}_{L^2(\Omega)}.
    \end{align*}
Integration in $s$ gives
    \begin{align*}
\norm{\phi|_{t=0}}_{H^1(\Omega)} 
+
\norm{\p_t\phi|_{t=0}}_{L^2(\Omega)} 
\lesssim
\norm{\phi}_{H^1(M)},
    \end{align*}
and the regularity (\ref{phi_reg_bd})
follows now from \cite{LLT}. 
It also follows from \cite{LLT} that 
    \begin{align*}
u \in C(0,T; L^2(\Omega)) \cap C^1(0,T; H^{-1}(\Omega)).
    \end{align*}
 
In the case that $u$ and $\phi$ are smooth
    \begin{align*}
0 = (P u, \phi)_{L^2(M)} - (u, P \phi)_{L^2(M)}
= -(\chi \p_\nu \phi, \p_\nu \phi)_{L^2(\Gamma)},
    \end{align*}
and for $(u, \phi) \in L^2(M) \times H^1(M)$ this can be justified by approximating $u$ and $\phi$ with smooth functions as in the proof of Lemma \ref{lem_uniq}. It follows from the boundary observability estimate, see Theorem \ref{th_obs_bd},
that $\phi = 0$ identically, and hence also $u=0$ identically.
\end{proof}

\subsection{Discretization}

Let us consider a family $\hat{\mathcal T} = \{\hat{\mathcal{T}}_h : h > 0\}$ where $\hat{\mathcal{T}}_h$ is a set of $1+n$-dimensional simplices forming a simplicial complex. The family $\hat{\mathcal T}$ is parametrized by 
    \begin{align*}
h = \max_{K \in \hat{\mathcal T}_h} \diam(K).
    \end{align*}
Writing $M_h = \bigcup_{K \in \hat{\mathcal{T}}_h} K$,
we assume that $M \subset M_h$. 
We define 
    \begin{align*}
\mathcal T_h = \{K \cap M : K \in \mathcal{T}_h\},
\quad \mathcal T = \{\mathcal T_h : h > 0 \},
    \end{align*}
and require that:
\begin{itemize}
\item[(T)] There is $C > 0$ such that for all $h> 0$ and all 
$K \in \mathcal{T}_h$, letting $\hat K \in \hat{\mathcal{T}}_h$
satisfy $K = \hat K \cap M$, there are balls $B_1 \subset K$ and $B_2 \supset \hat K$
such that the radii $r_j$ of $B_j$, $j=1,2$, satisfy
    \begin{align}\label{T_unif}
C^{-1} r_2 \le h \le C r_1,
    \end{align}
and that 
    \begin{align}\label{T_reg}
\nu(y) \cdot \rho(y) > C^{-1}, \quad \text{for all $y \in \p K$},
    \end{align} 
where $\nu$ is the outer unit normal vector of $\p K$, and $\rho(y) = (y-x)/|y-x|$ with $x$ is the centre of $B_1$.
\end{itemize} 

If $M$ was polyhedral, then we could choose $\hat{\mathcal T}$ so that $M_h = M$ for all small enough $h > 0$. In this case (T) follows if $\hat{\mathcal T}$ is quasi-uniform, see \cite[Definition 1.140]{Ern2004}. In the case of smooth $\Omega$, we can construct $\hat{\mathcal T}$ so that (T) holds for all small enough $h > 0$ by choosing polyhedral sets $M_h \supset M$ that approximate $M$ in the sense that the Hausdorff distance between $\p M_h$ and $\p M$ is of order $h^{1+\epsilon}$ for some $\epsilon > 0$, and meshing $M_h$ in a quasi-uniform manner. 

We define for $p \in \Np = \{1,2,\dots\}$ the $H^1(M)$-conformal approximation space of polynomial degree $p$,
 \begin{align}
\label{def_Vh}
V^p_{h} = \{u \in H^1(M) : u |_K \in \mathbb{P}_p(K)
\text{ for all $K \in \mathcal T_h$} \},
\end{align}
where $\mathbb{P}_p(K)$ denotes the set of polynomials of degree less
than or equal to $p$ on $K$. 
We write also $V_h = \bigcup_{p \in \Np} V_h^p$.
Note that, contrary to (\ref{def_Vh0}) no boundary condition is imposed on $\Gamma$.

The following two lemmas are proven in Appendix \ref{appendix_bd}.

\begin{lemma}\label{lem_trace_bd}
The trace inequality \eqref{trace_cont} holds for the family $\mathcal{T}$.
\end{lemma} 
\begin{lemma}
There is a family of interpolation operators $i^p_h : H^1(M) \to V_h^p$ satisfying
    \begin{align}\label{interp2}
\norm{u- i_h^p u}_{H^k(\mathcal T_h)}
\lesssim h^k \norm{u}_{H^k(M)}.
    \end{align}
\end{lemma}  

For any $h>0$, the control problem (\ref{control_eqs_bd}) can be formulated weakly as 
    \begin{align}\label{weak_form_bd}
a(u,\psi) = -c(h^{-1}\phi, \psi) + L(\psi), \quad a(v,\phi) = 0,
    \end{align}
for all $v,\psi \in C^\infty(M)$, where
\HOX{Lauri: The signs are a pain. I got the opposite sign in front of
  $c$, compared to the previous version.}
    \begin{align}\label{def_a2}
a(u,\psi) &= \int_{M} g(hdu, hd\psi) \d x + h^2(u, V \psi)_{L^2(M)}
\\\notag&\qquad
-h(u, h\p_\nu \psi)_{L^2(\p M)} - h(h\p_\nu u, \psi)_{L^2(\Gamma)}
\\\notag
c(\phi, \psi) &= h (\chi h \p_\nu \phi, h \p_\nu \psi)_{L^2(\Gamma)},
    \end{align}
and $L$ is as in (\ref{def_a1}).
Indeed, it follows from (\ref{int_by_parts}) that if smooth $(u,\phi)$ solves (\ref{control_eqs_bd})
then (\ref{weak_form}) holds for all smooth $(v,\psi)$.
We emphasize that $a$ and $c$ are chosen here so that they satisfy the continuity estimate (\ref{bilinear_bound}).

Our finite element method has the form:
find the critical point of the Lagrangian
    \begin{align*}
\mathcal L(u,\phi) : V_h^p \times V_h^q \to \R, \quad 
\mathcal L(u,\phi) = 
\frac12 c(\phi, \phi) + L(\phi)
- \frac12 \mathcal R(u,\phi)
+ a(u, \phi),
    \end{align*}
where, writing $U = (u, \p_t u)$ and $U_0 = (u_0, u_1)$,
the regularization is given by
    \begin{align*}
\mathcal R(u,\phi)
&= h^{-\kappa} S(u) - h^\kappa S(\phi) 
+ h^{-\kappa} E(U|_{t=0} - U_0)
+ h^{-\kappa} E(U|_{t=T}) 
\\&\qquad
+ \gamma h^{-\kappa} B(u) - h^\kappa B(\phi) 
+ \gamma h^{-\kappa} \widetilde C(\phi) + 2\gamma h^{-\kappa}\rho(u, \phi),
\\
B(u) &= h \norm{u}_{L^2(\Gamma)}^2, \quad
\widetilde C(\phi) = 
h\norm{\chi h\p_\nu \phi}_{L^2(\Gamma)}^2,
\quad
\rho(u,\phi) = - h(u, \chi h\p_\nu \phi)_{L^2(\Gamma)},
    \end{align*}
where $\kappa \le 0$ and $\gamma \in (0,1)$ are fixed constants.
Here $E$ and $S$ are as in (\ref{def_R})
except that $\Box$ in $S$ is replaced by $P$.

We have $\mathcal R(u,h^{-1}\phi) = 0$ for a smooth solution $(u,\phi)$ to 
(\ref{control_eqs_bd}).
Indeed, 
    \begin{align*}
B(u) + 2 \rho(u, h^{-1}\phi) + \widetilde C(h^{-1}\phi)
= 
h \norm{u - \chi \p_\nu \phi}^2_{L^2(\Gamma)}
= 0,
    \end{align*}
and also $S(u) = S(\phi) = B(\phi) = 0$ and $E(U|_{t=0} - U_0) = E(U|_{t=T}) = 0$.
The equation $d\mathcal L(u, \phi) = 0$ can be written as (\ref{fem})
where the bilinear form $A$ is now given by
    \begin{align*}
A[(u, \phi), (v, \psi)] 
&= 
h^{-\kappa}s(u,v) - h^\kappa s(\phi, \psi)
- h^2 c(\phi, \psi) 
+ h^{-\kappa} \sum_{\tau = 0,T} e(U|_{t=\tau}, V|_{t=\tau}) 
\\&\qquad + \gamma b(u,v) - b(\phi,\psi) - a(v, \phi) - a(u, \psi)
\\&\qquad+ \gamma h^{-\kappa} \tilde c(\phi, \psi) + 2 \gamma h^{-\kappa}\rho(v, \phi) + 2 \gamma h^{-\kappa}\rho(u, \psi).
    \end{align*}
We define the residual norm by
    \begin{align*}
\tnorm{(u,\phi)}^2 = 
h^{-\kappa} (S(u) + B(u)) 
+ h^\kappa (S(\phi) + B(\phi))
+ C(\phi) 
+ h^{-\kappa} \sum_{\tau = 0,T} E(U|_{t=\tau}). 
    \end{align*}
This is indeed a norm on $V_h \times V_h$ as can be seen by following the proof of Lemma \ref{lem_tnorm_is_norm}. Observe that in this case the vanishing boundary conditions on $\Gamma$ are not imposed in the spaces $V_h$ but follow if $B(u) = B(\phi) = 0$.

\begin{lemma}\label{lem_stability_bd}
For all $u,\phi \in H^2(M) + V_h$ there holds
    \begin{align*}
\tnorm{(u,\phi)}^2 \lesssim A[(u,\phi),(u,-\phi)].
    \end{align*}
\end{lemma} 
\begin{proof}
By the definition of $A$, we have
    \begin{align*}
A[(u, \phi),(u, -\phi)] 
=& \tnorm{(u,\phi)}^2 - \gamma h^{-\kappa} \widetilde C(\phi).
    \end{align*}
As $\kappa \le 0$, $\gamma < 1$ and $\chi \le 1$, $\gamma h^{-\kappa} \widetilde C(\phi)$ can be absorbed by $C(\phi)$.
\end{proof}

The previous two lemmas imply that the finite dimensional linear  system (\ref{fem}) has a unique solution,
and thus defines a finite element method. 

\subsection{Error estimates}
Equation (\ref{fem}) defines a finite element method that is 
consistent in the sense that if
smooth enough $u$ and $\phi$ satisfy (\ref{control_eqs}), then (\ref{fem}) holds for $(u,\phi)$. 
This follows from the weak formulation (\ref{weak_form}) of (\ref{control_eqs}) together with the regularization vanishing for $(u,\phi)$.
If smooth enough $u$ and $\phi$ satisfy (\ref{control_eqs_bd})
and if $(u_h, \phi_h) \in V_h^p \times V_h^q$ solves (\ref{fem})
then the Galerkin orthogonality 
\begin{equation}\label{eq:gal_ortho}
A[(u - u_h,h^{-1}\phi - \phi_h),(v,\psi)] = 0 \quad \mbox{ for all } (v,\psi)
\in V_h^p \times V_h^q.
\end{equation}
Analogously to the case of distributed control, 
this is due to $(u, \phi)$ satisfying the weak formulation (\ref{weak_form_bd}) and the regularization vanishing at $(u, h^{-1} \phi)$.

It is straightforward to see that for all $u,\phi,v,\psi \in H^2(M) + V_h$ there holds
    \begin{align*}
A[(u, \phi), (v, \psi)] 
+( a(v, \phi) + a(u, \psi) ) 
\lesssim 
\tnorm{(u,\phi)} \tnorm{(v,\psi)}.
    \end{align*}
We will need the following analogue of Lemma \ref{lem_a_cont}. We omit the proof, this being a modification of the earlier proof. The only difference is that the boundary terms on $\Gamma$ need to be kept track of.  

\begin{lemma}\label{lem_a_cont_bd}
For all $u,\phi,v,\psi \in H^2(M) + V_h$ there holds
    \begin{align*}
a(v, \phi) & \lesssim \left(S^\frac12(\phi) + B(\phi) \right)\norm{v}_{H^2(\mathcal T_h)},
\\
a(u, \psi) & \lesssim \left(S^\frac12(u) + B(u) + \sum_{\tau = 0,T} E^\frac12(U|_{t=\tau})\right) \norm{\psi}_{H^2(\mathcal T_h)}.
    \end{align*} 
\end{lemma}  

By repeating the proof of Theorem \ref{th:res_conv} we obtain:

\begin{theorem}\label{th:res_conv_bd}
Suppose that (A') holds.
Let $\kappa \le 0$, $p,q \in \Np$ and let $(u_h, \phi_h)$ in $V^p_h \times V^q_h$ be the solution of \eqref{fem}. Let $u \in H^{p+1}(M)$ and $\phi \in H^{q+1}(M)$ solve \eqref{control_eqs_bd}. Then
\[
\tnorm{(u-u_h,h^{-1}\phi-\phi_h)} 
\lesssim 
h^{p+1-\frac \kappa 2} \|u\|_{H^{p+1}(M)}
+
h^{q+1+\frac \kappa 2} \|h^{-1}\phi\|_{H^{q+1}(M)}.
\]
In particular,
    \begin{align}\label{est_control_bd}
\norm{\chi \p_\nu (\phi-h\phi_h)}_{L^2(\Gamma)} 
\lesssim  
h^{p-\frac \kappa 2 + \frac12} \|u\|_{H^{p+1}(M)}
+
h^{q+\frac \kappa 2 - \frac12} \|\phi\|_{H^{q+1}(M)}.
    \end{align}
\end{theorem}

As discussed above, 
 if $(u_0,u_1) \in H^{k+1}(\Omega)\times H^{k}(\Omega)$
satisfies the compatibility conditions of order $k$,
then the solution to (\ref{control_eqs_bd}) satisfies 
$$(u,\phi) \in H^{k+1}(M) \times H^{k+2}(M).$$
Hence we can take $\kappa = 0$, $q \le k+1$ and $p=q-1$ in the above theorem, leading to the convergence rate (\ref{conv_intro_bd}) stated in the introduction.

\begin{remark}\label{rem_dist_bd}
A finite element method for the distributed control problem
can be formulated using the spaces $V_h^p$ defined by (\ref{def_Vh})
and the bilinear form $a$ in (\ref{def_a2}).
With these choices replacing $V_h^p$ and $a$ in the Lagrangian in Section \ref{sec_disc}, and with $B(u) - B(\phi)$ added in the regularization $\mathcal R(u,\phi)$ there, 
we obtain a method satisfying the estimates in Theorem \ref{th:res_conv}. This method works for smooth $\Omega$ whenever the geometric control condition (A) holds. We omit proving this, the proof being very similar with those above. 
\end{remark}

\section{Numerical experiments}

We discuss some numerical experiments performed with the Freefem++ package (see \cite{hecht2012}).

We address the distributed and boundary case in the one dimensional case and emphasize the influence of the regularity of the initial condition on the rate of convergence of the finite element method with respect to the size of the discretization.  We use uniform (unstructured) meshes and the cut off functions $\chi_0\in C_0^\infty([0,T])$ and $\chi_1\in C_0^\infty([0,1])$ defined as follows
\begin{equation}
\label{chi_used}
\chi_0(t)=\frac{e^{-\frac{1}{2t}}e^{-\frac{1}{2(T-t)}}}{e^{-\frac{1}{T}}e^{-\frac{1}{T}}}, \qquad \chi_1(x)=\frac{e^{-\frac{1}{5(x-a)}}e^{-\frac{1}{5(b-x)}}}{e^{-\frac{2}{5(b-a)}}e^{-\frac{2}{5(b-a)}}} 1_{[a,b]}(x)
\end{equation}
for any $0<a<b<1$ and $T>0$. In particular, $\chi_0(T/2)=1$  and $\chi_1((a+b)/2)=1$. Figure $\ref{chi}$ depicts the function $\chi_0$ for $T=2.5$.

\begin{figure}[!http]
\begin{center}
\includegraphics[scale=0.4]{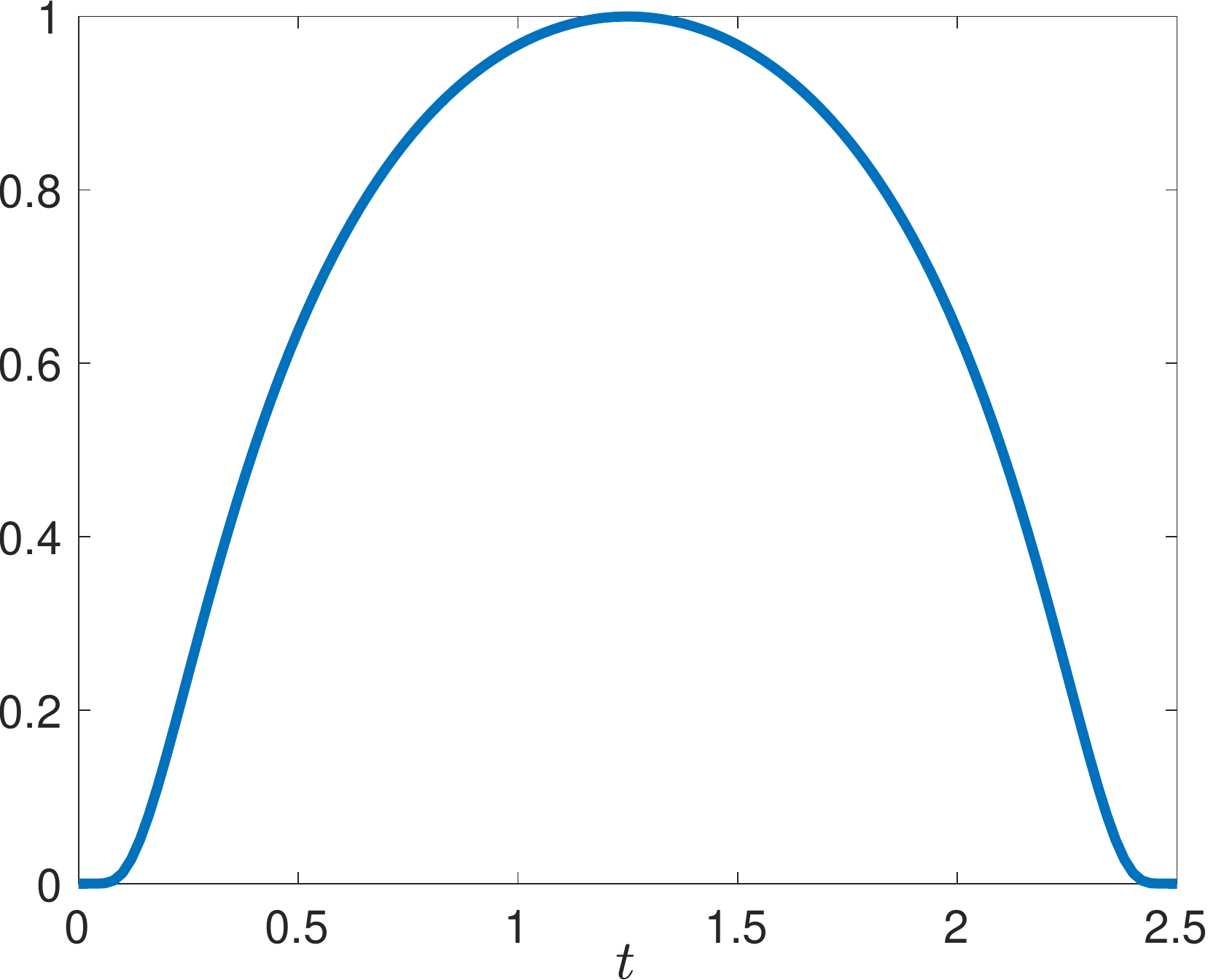}
\caption{The $C_0^\infty([0,T])$ function $t\mapsto \chi_0(t)$, $t\in [0,T]$ with $T=2.5$.}\label{chi}
\end{center}
\end{figure}

\subsection{Distributed case : initial condition in $H^{k+1}(\Omega)\times H^{k}(\Omega)$ for all $k\in \mathbb{N}$}
We consider the simplest situation for which 
\begin{equation}
\label{eq:Ex1}\tag{\bf Ex1}
 (u_0,u_1)=(\sin(\pi x),0)\in H^{k+1}(\Omega)\times H^{k}(\Omega) \quad \forall k\in \mathbb{N}.
\end{equation}
Compatibility conditions \eqref{C1L}-\eqref{C2L} are satisfied for any $j$. Moreover, we use the cut-off functions $\chi_0\in C_0^\infty([0,T])$ and $\chi_1\in C_0^\infty([0,1])$ defined by \eqref{chi_used} with $T=2$, $a=0.1$ and $b=0.4$. The null controllability property (A) holds true for this set of data.
Since explicit solutions are not available in the distributed case, we define as ``exact'' solution $(u,\phi)$ the one of \eqref{fem}  from a fine and structured mesh (composed of $409\ 000$ triangles and $205\ 261$ vertices) corresponding to $h\approx 4.41\times 10^{-3}$   and $(u_h,\phi_h)\in V_h^p\times V_h^q$ with $(p,q)=(3,3)$.

Figure \ref{fig:ex1_IN_avecchi}-left depicts the evolution of the relative error for the variable $\phi$ with respect to the $L^2$-norm
$$
\textrm{err}(\phi,\phi_h,\chi):= \Vert\chi(\phi-\phi_h)\Vert_{L^2(M)}/\Vert\chi \phi\Vert_{L^2(M)}
$$ 
with respect to $h$ for various pairs of $(p,q)$. Table \ref{tab:ex1_IN_CHI} collects the corresponding numerical values. 
We observe the convergence of the approximation w.r.t. $h$. Moreover, the figure exhibits the influence of the space $V_h^q$ used for the variable $\phi_h$ while the choice of the space $V_h^p$ for the variable $u_h$ has no effect on the approximation. We observe rates close to $0.5$, $2$ and $3$ for $(p,q)=(1,1)$, $(p,q)=(2,2)$ and $(p,q)=(3,3)$ respectively, in agreement with Theorem  \ref{th:res_conv}.  For comparison, Figure \ref{fig:ex1_IN_avecchi}-right depicts the evolution of the relative error $\textrm{err}(\phi,\phi_h,\chi)$ for $\chi_0(t)=1$ and $\chi_1(x)=1_{(a,b)}(x)$, i.e. when no regularization of the control support is introduced. Table \ref{tab:ex1_IN_sanschi} collects the corresponding numerical values. The corresponding controlled pair $(u,\phi)$ is a priori only in $C([0,T];H_0^1(\Omega))\times C([0,T];L^2(\Omega))$. Thus, if we still get the convergence with respect to the parameter $h$, we observe that the approximation is not improved beyond the value $q=2$. As before, the choice of the approximation space $V_h^p$ for $u_h$ does not affect the result. The rate is also reduced: for $(p,q)=(2,2)$, the rate is close to $1.5$. This highlights the influence of the cut off functions, including for very smooth initial conditions.  
Table \ref{tab:ex1_P2P2} collects some $L^2$ norms of $u_h$ and $\phi_h$ with respect to $h$ for the pair $(p,q)=(2,2)$: in particular, the relative error $err(u,u_h,1)$ associated with the controlled solution $u$ is order of $h^{2.5}$ for $h$ small enough.

\begin{figure}[!http]
\begin{center}
\includegraphics[scale=0.37]{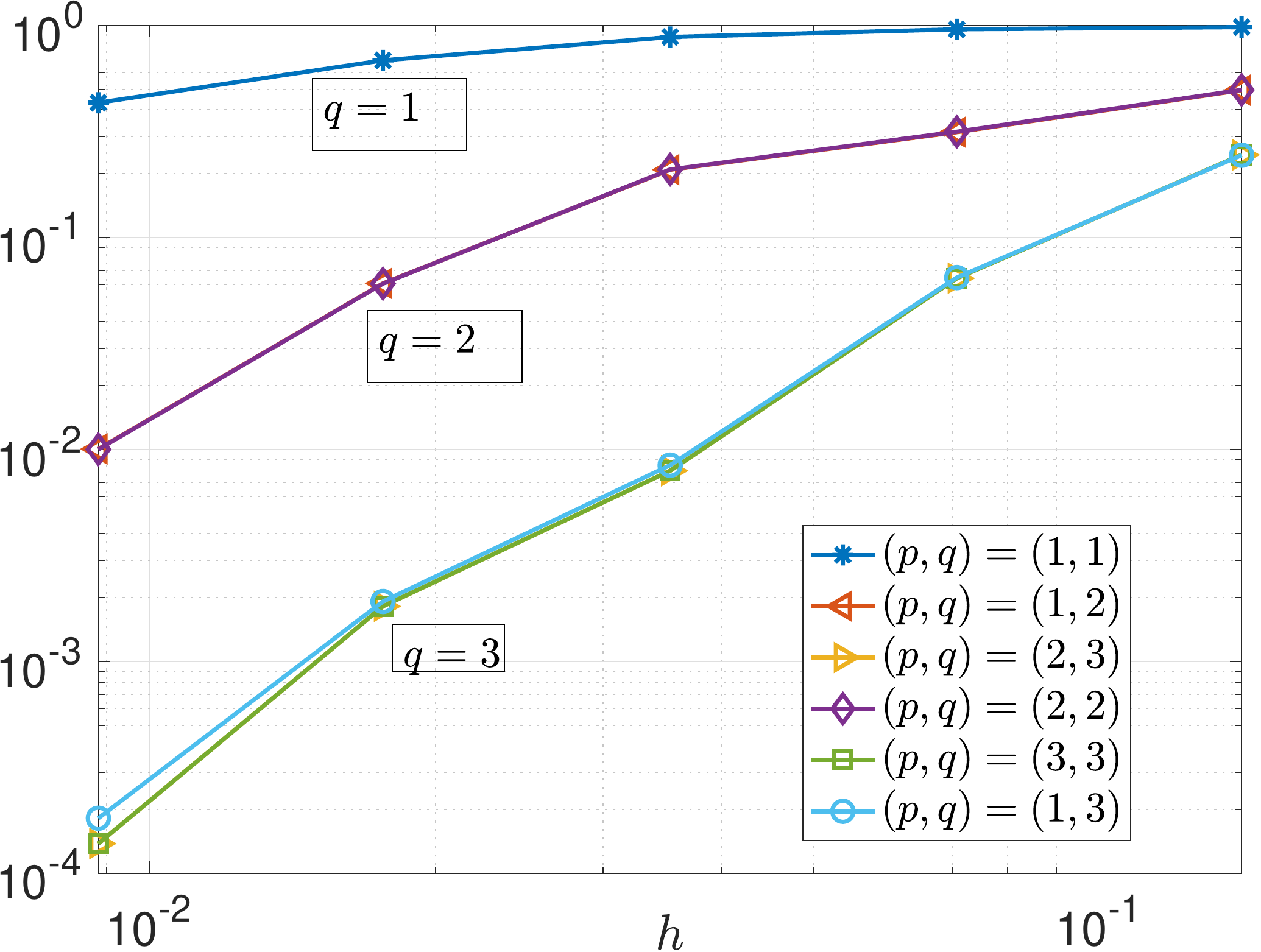}
\includegraphics[scale=0.43]{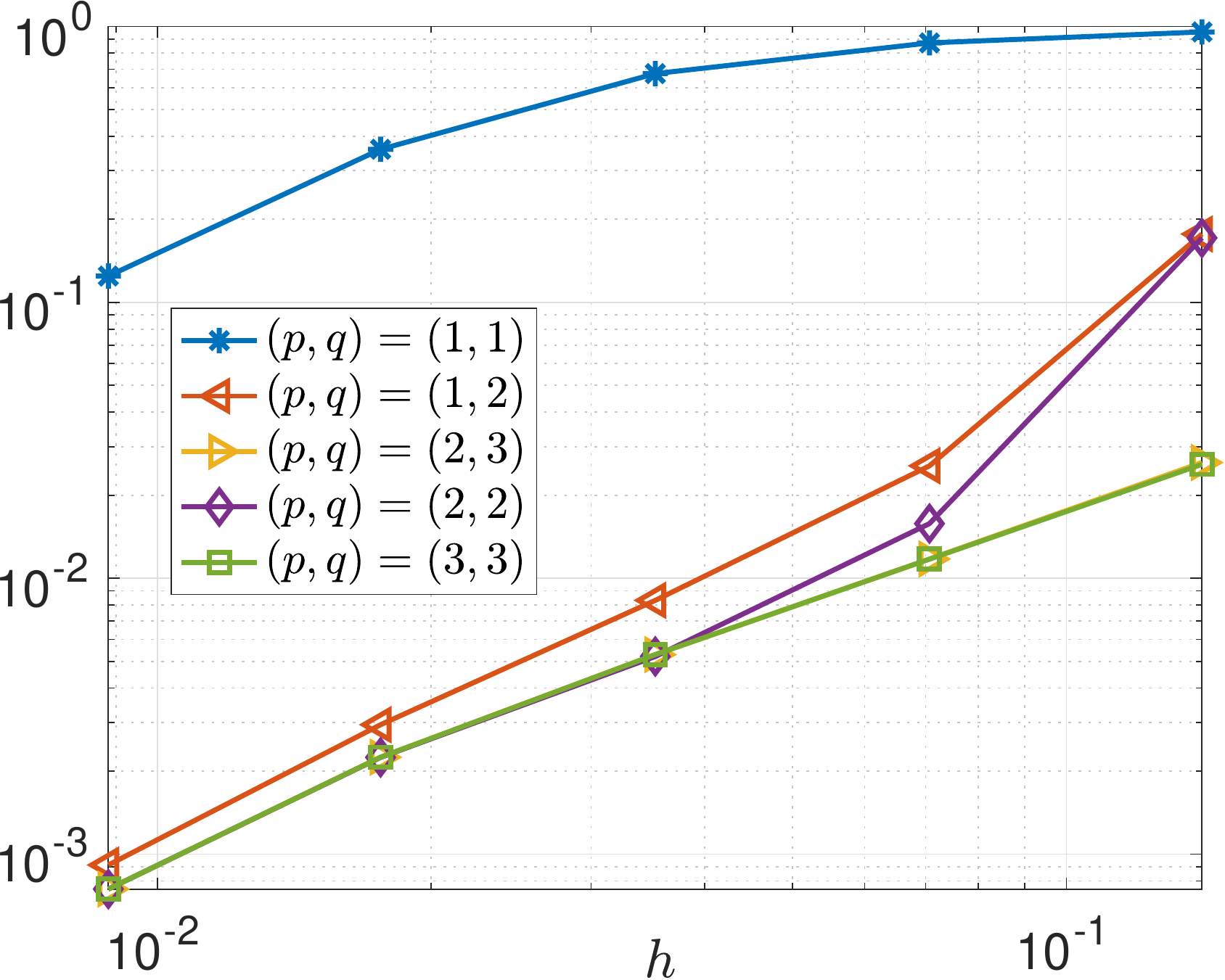}
\caption{\eqref{eq:Ex1}; $\Vert\chi(\phi-\phi_h)\Vert_{L^2(M)}/\Vert\chi \phi\Vert_{L^2(M)}$ vs. $h$; Left: $\chi_0(t)$ and $\chi_1(x)$ given by \eqref{chi_used}; Right: $\chi_0(t)=1$ and $\chi_1(x)=1_{(a,b)}(x)$.}\label{fig:ex1_IN_avecchi}
\end{center}
\end{figure}

\begin{table}[http!]
	\centering
		\begin{tabular}{|l|ccccc|}
			\hline
			 			   $h$  & $1.57\times 10^{-1}$  &      $8.22\times 10^{-2}$  	&  $4.03\times 10^{-2}$ 	&  $2.29\times 10^{-2}$  	&  $1.25\times 10^{-2}$  
						   			\tabularnewline
			\hline
			$(p,q)=(1,1)$ 	  & $9.81\times 10^{-1}$ &	$9.58\times 10^{-1}$ &	$8.81\times 10^{-1}$ &	$6.83\times 10^{-1}$ &	$4.31\times 10^{-1}$			
			\tabularnewline
			$(p,q)=(2,2)$ 	  & $4.96\times 10^{-1}$	& $3.15\times 10^{-1}$ &	$2.09\times 10^{-1}$ 	& $6.05\times 10^{-2}$ & $1.00\times 10^{-2}$
			\tabularnewline
			$(p,q)=(3,3)$ 	  & $2.45\times 10^{-1}$	& $6.41\times 10^{-2}$	& $7.93\times 10^{-3}$ &	$1.81\times 10^{-3}$ &	$1.38\times 10^{-4}$
          		\tabularnewline
			\hline
		\end{tabular}
		\vspace{0.1cm}
	\caption{\eqref{eq:Ex1}; $\Vert \chi(\phi_h-\phi)\Vert_{L^2(M)}/\Vert \chi \phi\Vert_{L^2(M)}$; $\chi$ from \eqref{chi_used}.}
	\label{tab:ex1_IN_CHI}
\end{table}

\begin{table}[http!]
	\centering
		\begin{tabular}{|l|ccccc|}
			\hline
			 			   $h$  & $1.57\times 10^{-1}$  &      $8.22\times 10^{-2}$  	&  $4.03\times 10^{-2}$ 	&  $2.29\times 10^{-2}$  	&  $1.25\times 10^{-2}$  
						   			\tabularnewline
			\hline
			$(p,q)=(1,1)$ 	  & $9.55\times 10^{-1}$  & $8.71\times 10^{-1}$ & $6.74\times 10^{-1}$ & $3.58\times 10^{-1}$ & $1.24\times 10^{-1}$  			
			\tabularnewline
			$(p,q)=(2,2)$ 	  & $1.71\times 10^{-1}$ &	$1.57\times 10^{-2}$ &	$5.20\times 10^{-3}$	& $2.24\times 10^{-3}$ &	$7.47\times 10^{-4}$
			\tabularnewline
			$(p,q)=(3,3)$ 	  & $2.58\times 10^{-2}$ &	$1.17\times 10^{-3}$	 & $5.28\times 10^{-3}$ 	& $2.24\times 10^{-3}$	& $7.47\times 10^{-4}$
			
			\tabularnewline
			\hline
		\end{tabular}
		\vspace{0.1cm}
	\caption{\eqref{eq:Ex1}; $\Vert \chi(\phi_h-\phi)\Vert_{L^2(\mathcal{M})}/\Vert \chi \phi\Vert_{L^2(\mathcal{M})}$ ; $\chi_0(t)=1; \chi_1(x)=1_{(a,b)}(x)$.}
	\label{tab:ex1_IN_sanschi}
\end{table}

\begin{table}[http!]
	\centering
		\begin{tabular}{|l|ccccc|}
			\hline
			 			   $h$  & $1.57\times 10^{-1}$  &      $8.22\times 10^{-2}$  	&  $4.03\times 10^{-2}$ 	&  $2.29\times 10^{-2}$  	&  $1.25\times 10^{-2}$  
						   			\tabularnewline
			\hline
			$\textrm{err}(\phi,\phi_h,1)$ 	  & $6.22\times 10^{-1}$ &      $3.92\times 10^{-1}$ 	& $2.33\times 10^{-1}$	& $6.45\times 10^{-2}$ 		 		& $1.04\times 10^{-2}$	
			\tabularnewline
			$\textrm{err}(\partial_x \phi,\partial_x \phi_h,1)$ 	  & $8.11\times 10^{-1}$ &      $6.84\times 10^{-1}$ 	& $4.85\times 10^{-1}$	& $1.64\times 10^{-1}$ 		 		& $4.57\times 10^{-2}$	
			\tabularnewline
			$\textrm{err}(u,u_h,1)$ 	  & $4.29\times 10^{-1}$ &      $1.36\times 10^{-1}$ 	& $5.00\times 10^{-2}$	& $1.08\times 10^{-2}$ 		 		& $1.10\times 10^{-3}$	
			\tabularnewline
			$\textrm{err}(\phi,\phi_h,\chi)$ 	  & $4.96\times 10^{-1}$ &      $3.15\times 10^{-1}$ 	& $2.09\times 10^{-1}$	& $6.05\times 10^{-2}$ 		 		& $1.00\times 10^{-2}$	
			\tabularnewline
			\hline
		\end{tabular}
		\vspace{0.1cm}
	\caption{\eqref{eq:Ex1};  $(p,q)=(2,2)$ and $\chi$ from \eqref{chi_used}.}
	\label{tab:ex1_P2P2}
\end{table}

\subsection{Distributed case : initial condition in $H^1(\Omega)\times L^2(\Omega)$}
We consider the initial condition 
\begin{equation}
\label{eq:Ex2}\tag{\bf Ex2}
(u_0,u_1)=\big(4x 1_{(0,1/2)}(x)+4(1-x)1_{ [1/2,1)}(x),0\big)\in H^1(\Omega)\times H^{0}(\Omega)
\end{equation}
for which the compatibility conditions \eqref{C1L}-\eqref{C2L} are satisfied for any $j$. If the cut-off functions are introduced, the controlled pair $(u,\phi)$ belongs to $H^1(M)\times H^0(M)$. 
Theorem \ref{th:res_conv} does not provide a convergence rate in this case. The strong convergence is however observed: figure \ref{fig:ex2_IN_CHI} displays the relative error wrt $h$ for $(p,q)=(1,1)$, $(p,q)=(2,2)$ and $(p,q)=(3,3)$ 
with rates close to $1/2$. 
\begin{figure}[!http]
\begin{center}
\hspace*{-0.5cm}
\includegraphics[scale=0.45]{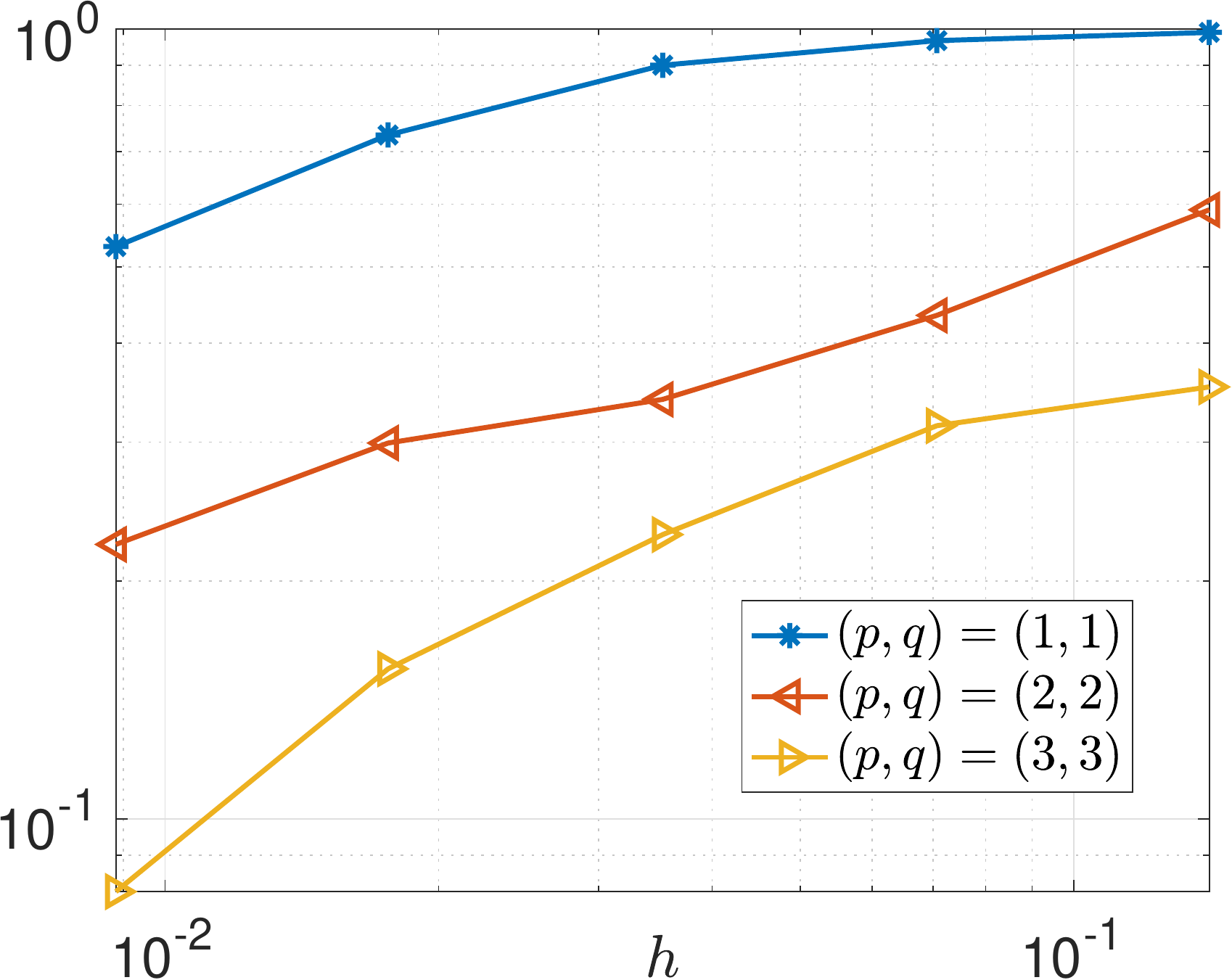}
\includegraphics[scale=0.45]{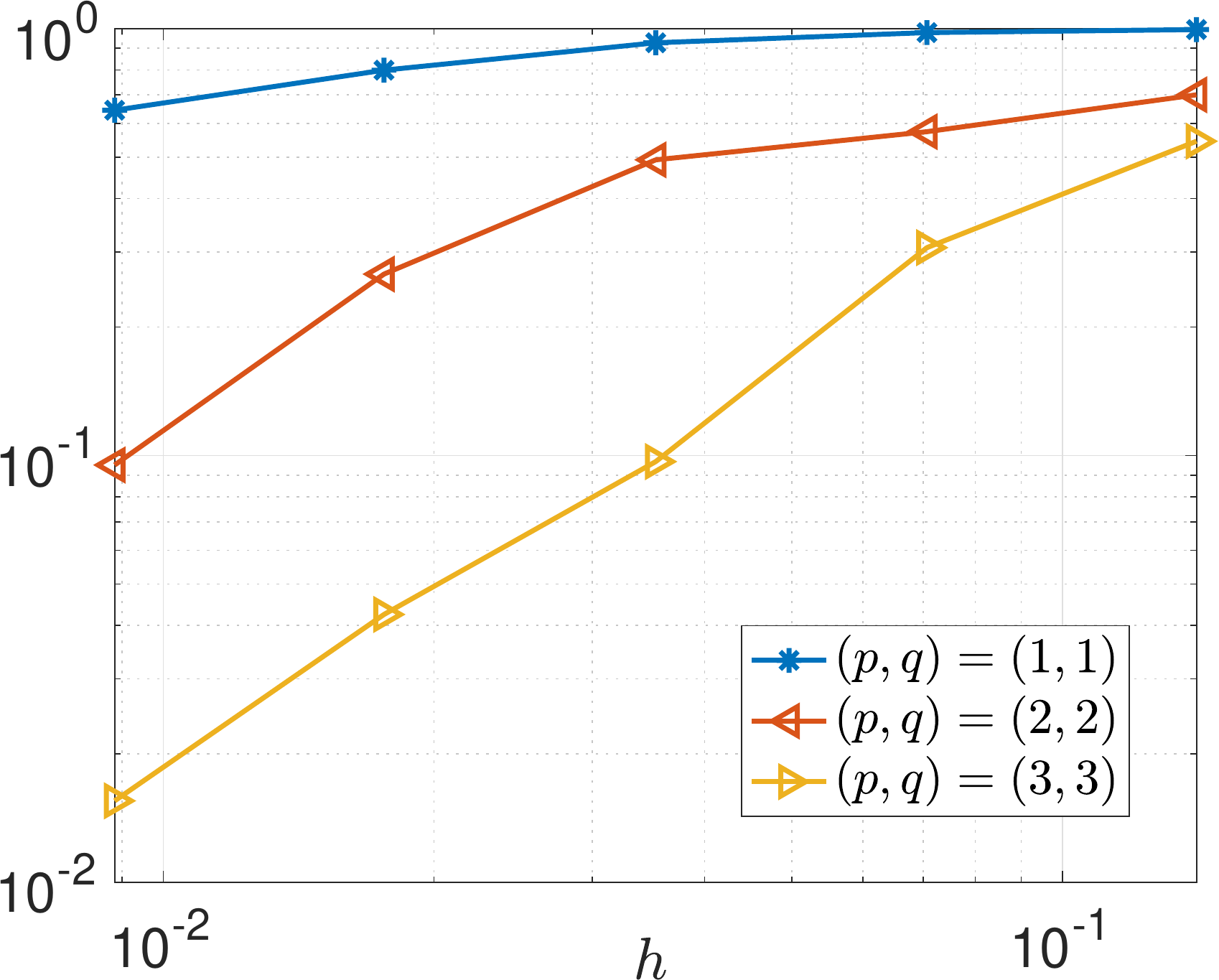}
\caption{\eqref{eq:Ex2} and \eqref{eq:Ex2b} ; $\Vert\chi(\phi-\phi_h)\Vert_{L^2(M)}/\Vert\chi \phi\Vert_{L^2(M)}$ vs. $h$; $\chi$ from \eqref{chi_used}.}\label{fig:ex2_IN_CHI}
\end{center}
\end{figure}

\begin{table}[http!]
	\centering
		\begin{tabular}{|l|ccccc|}
			\hline
			 			   $h$  & $1.57\times 10^{-1}$  &      $8.22\times 10^{-2}$  	&  $4.03\times 10^{-2}$ 	&  $2.29\times 10^{-2}$  	&  $1.25\times 10^{-2}$  
						   			\tabularnewline
			\hline
			$(p,q)=(1,1)$ 	  & $9.89\times 10^{-1}$	& $9.66\times 10^{-1}$	& $8.99\times 10^{-1}$ &	$7.34\times 10^{-1}$ &	$5.30\times 10^{-1}$			
			\tabularnewline
			$(p,q)=(2,2)$ 	  & $5.90\times 10^{-1}$	& $4.31\times 10^{-1}$ &	$3.38\times 10^{-1}$ &	$2.98\times 10^{-1}$ &	$2.22\times 10^{-1}$
			\tabularnewline
			$(p,q)=(3,3)$ 	  & $3.51\times 10^{-1}$	& $3.14\times 10^{-1}$	& $2.28\times 10^{-1}$ &	$1.54\times 10^{-1}$  &	$8.08\times 10^{-2}$
			
			\tabularnewline
			\hline
		\end{tabular}
		\vspace{0.1cm}
	\caption{\eqref{eq:Ex2}; $\Vert \chi(\phi-\phi_h)\Vert_{L^2(M)}/\Vert \chi \phi\Vert_{L^2(M)}$; $\chi$ from \eqref{chi_used}.}
	\label{tab:ex2_IN_CHI}
\end{table}
A similar behavior is observed with the condition $u_0=0$ and $u_1=1_{(0.4,0.6)}(x)$ in $L^2(\Omega)$ in agreement with Theorem \ref{th_rough}

\subsection{Distributed case : initial condition in $H^2(\Omega)\times H^1(\Omega)$}
We consider the initial condition
\begin{equation}
\label{eq:Ex2b}\tag{\bf Ex2b}
 (u_0,u_1)=\big(\rho(x)\int_0^x u_0^2(t)dt,0\big)\in H^2(\Omega)\times H^1(\Omega)
 \end{equation}
where $u_0^2$ is the initial position defined in \eqref{eq:Ex2}  and $\rho\in C^\infty_0(\Omega)$ is introduced in order to preserve the compatibility conditions \eqref{C1L}-\eqref{C2L}. 
Figure \ref{fig:ex2_IN_CHI}-right displays the convergence of the
approximation for $(p,q)=(1,1)$, $(p,q)=(2,2)$ and
$(p,q)=(3,3)$. Theorem \ref{th:res_conv} still does not provide a
convergence rate in this case. However, with respect to the previous example, smaller relative error with rates close to $1$ are observed for $(p,q)=(2,2)$ and $(p,q)=(3,3)$.

\subsection{Boundary case: initial condition in $H^{k+1}(\Omega)\times H^{k}(\Omega)$ for all $k\in \mathbb{N}$}
We consider again the simple situation given by the initial condition \eqref{eq:Ex1}. Compatibility conditions \eqref{C1L}-\eqref{C2L} are satisfied for any $j$. In contrast with the distributed case,  explicit exact solutions are available in the boundary case when cut-off functions are not introduced. Precisely, the corresponding control of minimal $L^2(\Gamma)$ norm with $\Gamma=(0,T=2)\times \{1\}$ is given by $v(t)=\frac{1}{2}\sin(\pi(1-t))=\frac{1}{2}\sin(\pi t)$ leading to $\Vert v\Vert_2=1/2$. The corresponding controlled solution is given by 
\begin{equation}
u(t,x)=\left\{
\begin{array}{cc}
 \frac{1}{2}\big(u_0(x+t)+u_0(x-t)\big) & x+t\leq 1, \\
 \frac{1}{2}u_0(x-t), & x+t>1, \,x-t\geq -1,\\
 0,  & x-t<-1,
 \end{array}
\right.
\end{equation}
leading to $\Vert u\Vert_{L^2(M)}=1/2$. The corresponding adjoint solution is 
given by $\phi(t,x)=-\frac{1}{2\pi}\sin(\pi t)\sin(\pi x)$ leading to $\Vert \phi\Vert_{L^2(M)}=\frac{1}{2\sqrt{2}\pi}$ and $\Vert \partial_x\phi\Vert_{L^2(M)}=\frac{1}{2\sqrt{2}}$.

Tables \ref{tab:ex1_12} and \ref{tab:ex1_23} collects some relative errors w.r.t. $h$ for $(p,q)=(1,2)$ and $(p,q)=(2,3)$ respectively including 
$$
\textrm{err}(v,u_h):=\Vert v-u_h(\cdot,1)\Vert_{L^2(0,T)}/\Vert v\Vert_{L^2(0,T)}, 
$$
and $\textrm{err}(v,h\chi\partial_{\nu}\phi_h)$ while Figure \ref{fig_ex1}-left depicts the relative error on the control w.r.t. $h$ for several pairs of $(p,q)$. 
Since compatibility conditions hold true here, the introduction of the cut off function $\chi\neq 1$ is a priori useless.  However, we observe that the term $\partial_x \phi(\cdot,1)$ is not well approximated near $t=0$ and $t=T$.
This somehow pollutes the approximation $\phi_h$ of $\phi$ inside the domain (precisely along the characteristics intersecting the points $(t,x)=(0,1)$ and $(t,x)=(T,1)$ and affects the optimal rate. We observe rate close to $0.75$ for $(p,q)=(1,1)$ and close to $1.5$ otherwise. Imposing in addition $\phi=0$ on the boundary $\partial\Omega$ slightly improves the approximation.

\begin{table}[http!]
	\centering
		\begin{tabular}{|l|ccccc|}
			\hline
			 			   $h$  & $1.57\times 10^{-1}$  &      $8.22\times 10^{-2}$  	&  $4.03\times 10^{-2}$ 	&  $2.29\times 10^{-2}$  	&  $1.25\times 10^{-2}$  
						   			\tabularnewline
			\hline
			$\textrm{err}(\phi,\phi_h,1)$ 	  & $7.17\times 10^{-1}$ &      $1.97\times 10^{-1}$ 	& $4.54\times 10^{-2}$	& $1.53\times 10^{-2}$ 		 		& $4.50\times 10^{-3}$	
			\tabularnewline
			$\textrm{err}(\partial_x\phi,\partial_x\phi_h,1)$  	  & $9.21\times 10^{-1}$ &      $3.05\times 10^{-1}$ 	& $9.40\times 10^{-2}$	& $5.41\times 10^{-2}$ 		 		& $2.47\times 10^{-2}$	
			\tabularnewline
			$\textrm{err}(u,u_h,1)$ 	  & $1.88\times 10^{-1}$ &      $5.73\times 10^{-2}$ 	& $1.92\times 10^{-2}$	& $9.29\times 10^{-3}$ 		 		& $4.09\times 10^{-3}$	
			\tabularnewline
			$\textrm{err}(v,u_h)$ 	  & $2.16\times 10^{-1}$ &      $5.83\times 10^{-2}$ 	& $1.94\times 10^{-2}$	& $7.21\times 10^{-3}$ 		 		& $3.10\times 10^{-3}$	
			\tabularnewline
			$\textrm{err}(v,h\chi\partial_x\phi_h)$ 	  & $2.71\times 10^{-1}$ &      $8.41\times 10^{-2}$ 	& $2.79\times 10^{-2}$	& $1.22\times 10^{-2}$ 		 		& $5.31\times 10^{-3}$	                
			\tabularnewline
			\hline
		\end{tabular}
		\vspace{0.1cm}
	\caption{\eqref{eq:Ex1} -  Boundary case -  $(p,q)=(1,2)$ - $\chi\equiv1$.}
	\label{tab:ex1_12}
\end{table}

\begin{table}[http!]
	\centering
		\begin{tabular}{|l|ccccc|}
			\hline
			 			   $h$  & $1.57\times 10^{-1}$  &      $8.22\times 10^{-2}$  	&  $4.03\times 10^{-2}$ 	&  $2.29\times 10^{-2}$  	&  $1.25\times 10^{-2}$  
						   			\tabularnewline
			\hline
			$\textrm{err}(\phi,\phi_h,1)$ 	  & $4.09\times 10^{-1}$ &      $1.15\times 10^{-1}$ 	& $2.82\times 10^{-2}$	& $9.29\times 10^{-3}$ 		 		& $2.81\times 10^{-3}$	
			\tabularnewline
			$\textrm{err}(\partial_x\phi,\partial_x\phi_h,1)$ 	  & $6.73\times 10^{-1}$ &      $2.10\times 10^{-1}$ 	& $6.13\times 10^{-2}$	& $2.91\times 10^{-2}$ 		 		& $1.24\times 10^{-2}$	
			\tabularnewline
			$\textrm{err}(u,u_h,1)$ 	  & $7.29\times 10^{-2}$ &      $2.81\times 10^{-2}$ 	& $9.02\times 10^{-3}$	& $3.45\times 10^{-3}$ 		 		& $1.28\times 10^{-3}$	
			\tabularnewline
			$\textrm{err}(v,u_h)$ 	  & $1.00\times 10^{-1}$ &      $3.74\times 10^{-2}$ 	& $1.13\times 10^{-2}$	& $4.16\times 10^{-3}$ 		 		& $1.43\times 10^{-3}$	
			\tabularnewline
			$\textrm{err}(v,h \chi \partial_x\phi_h)$ 	  & $1.05\times 10^{-1}$ &      $3.89\times 10^{-2}$ 	& $1.23\times 10^{-2}$	& $4.61\times 10^{-3}$ 		 		& $1.65\times 10^{-3}$	 
			\tabularnewline
			\hline
		\end{tabular}
		\vspace{0.1cm}
	\caption{\eqref{eq:Ex1} -  Boundary case -  $(p,q)=(2,3)$  - $\chi\equiv 1$.}
	\label{tab:ex1_23}
\end{table}

\begin{figure}[!http]
\begin{center}
\includegraphics[scale=0.45]{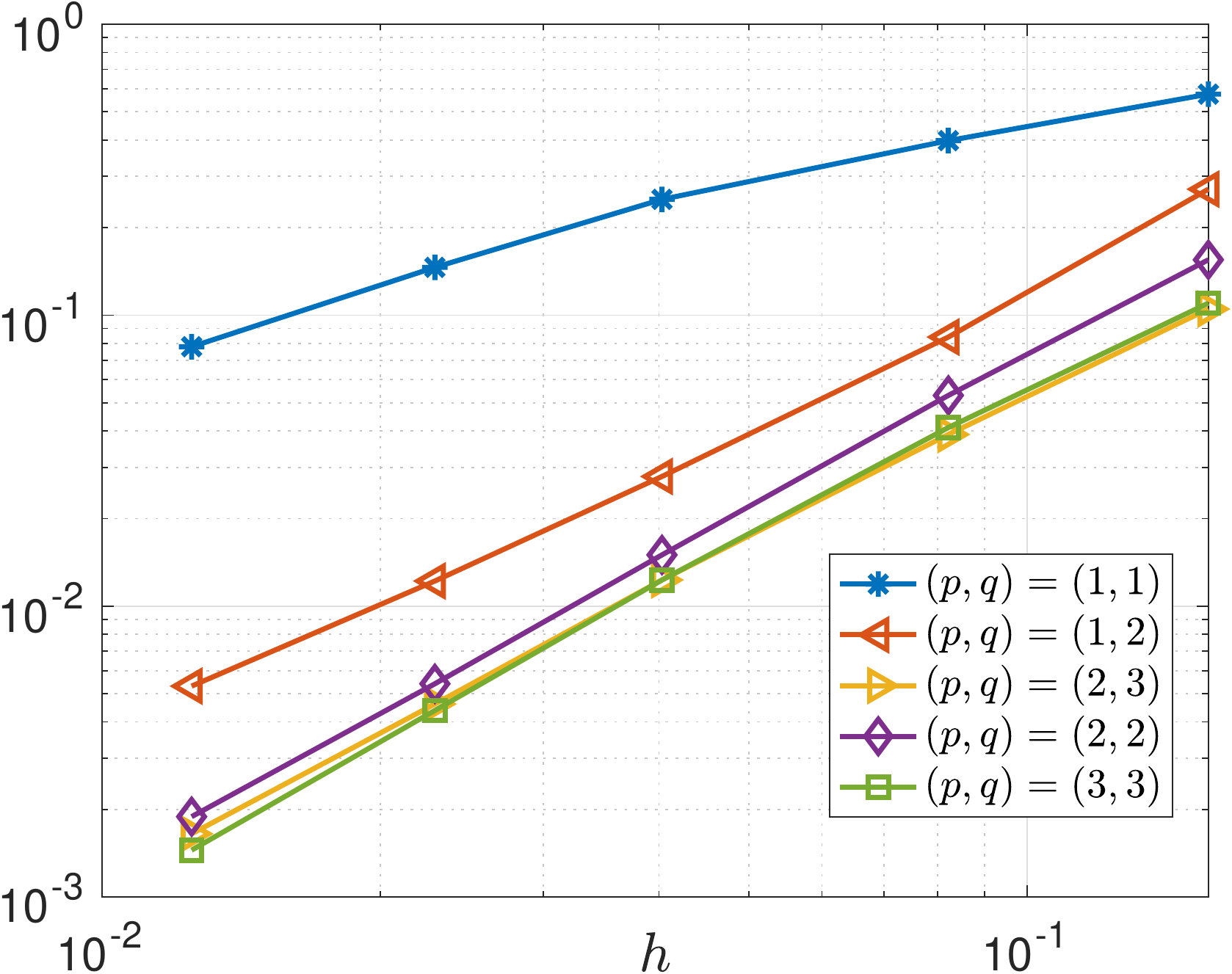}
\includegraphics[scale=0.45]{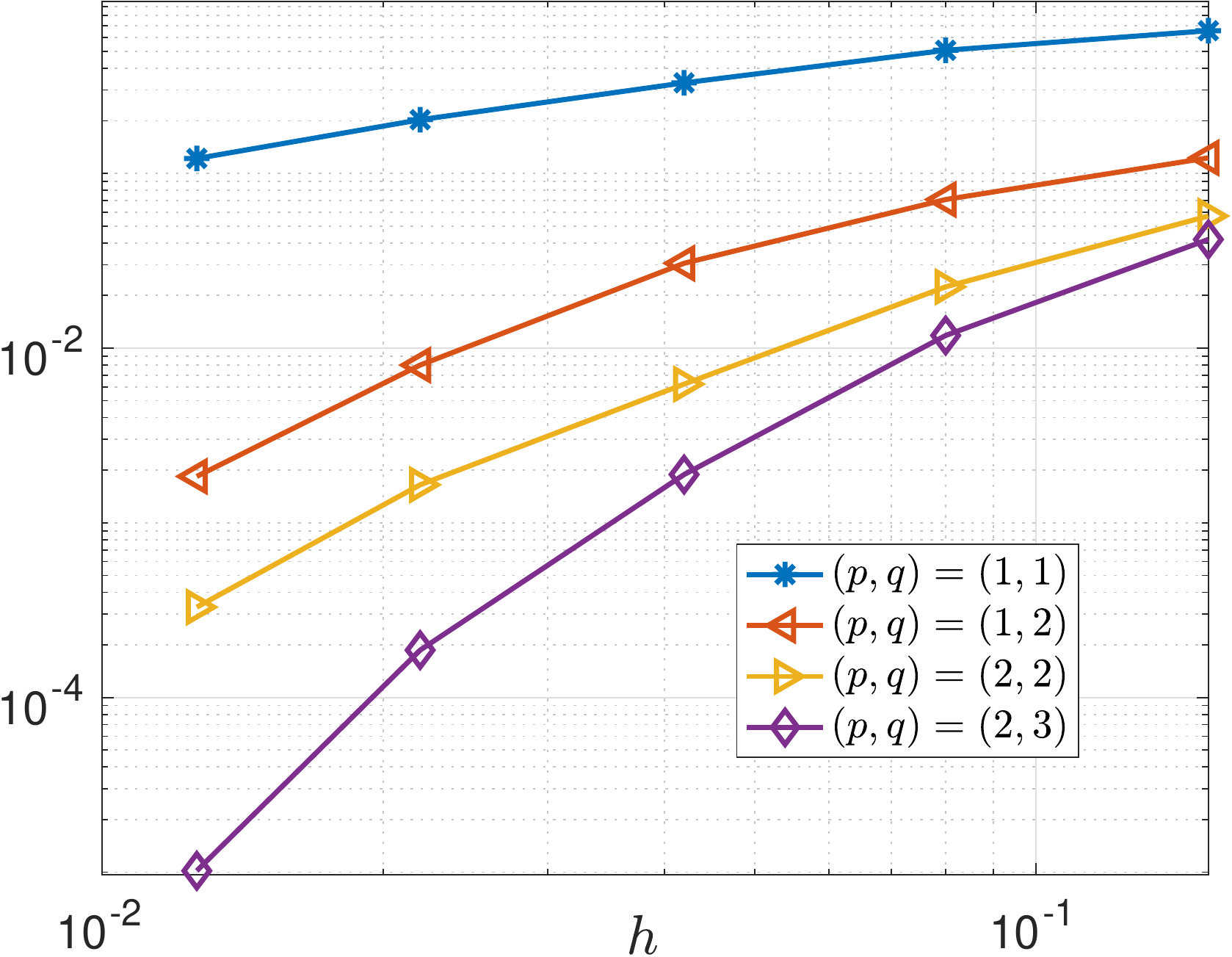}
\caption{\eqref{eq:Ex1} - Boundary case -  $\Vert h\chi\partial_x\phi_h(1,\cdot)-v\Vert_{L^2(0,T)}/\Vert v\Vert_{L^2(0,T)}$ vs. $h$ with $\chi=1$ (left) and $\chi=\chi_0$ from \eqref{chi_used} (right).}\label{fig_ex1}
\end{center}
\end{figure}

With $\chi\neq 1$, the explicit control of minimal $L^2(\chi^{1/2},(0,T))$ norm is not available anymore. As for the distributed controllability, we define as ``exact" control the one obtained from a fine and uniform mesh (composed of $648\ 000$ triangles and $325\ 261$ vertices) corresponding to $h\approx 3.92\times 10^{-3}$ and $(p,q)=(2,3)$. We take $T$ large enough, precisely $T=2.5$, to ensure the null controllability property of the wave equation. Figure \ref{fig_ex1}-right displays the evolution of $\Vert h \chi\partial_x\phi_h(\cdot,1)-v\Vert_{L^2(0,T)}/\Vert v\Vert_{L^2(0,T)}$ w.r.t. $h$.  For $(p,q)=(1,1)$ and $(p,q)=(1,2)$, we observe rates close to $0.5$ and $1.5$ in agreement with \eqref{est_control_bd} of Theorem \ref{th:res_conv_bd} with $\kappa=0$. For $(p,q)=(2,3)$, we observe a rate close to $3$, which is a bit better than the value $2.5$ from \eqref{est_control_bd}. Those results also show that the boundary control can be approximated both from the quantity $h\partial_x \phi_h(\cdot,1)$ obtained from the adjoint dual variable and from the trace $u_h(\cdot,1)$ of the primal variable.

\subsection{Boundary case: initial condition in $H^{1}(\Omega)\times H^{0}(\Omega)$}

We consider the initial data \eqref{eq:Ex2}. The corresponding control
of minimal $L^2(\Gamma)$ norm with $\Gamma= (0,T=2)\times \{1\}$ (corresponding to $\chi\equiv 1$)is given by 
 $$
 v(t)=2t \,1_{(0,1/2)}(t) +2(1-t) 1_{(1/2,3/2)}(t) +2(t-2)\,1_{(3/2,1)}(t).
 $$ 
 The corresponding controlled solution is explicitly known as follows: 
 \begin{equation}
 u(t,x)=\frac{1}{2}u_0(x-t)1_{(t-x\leq 0)}+ \frac{1}{2}u_0(x+t)1_{(t+x\leq 1)}-\frac{1}{2}u_0(t-x)1_{(t-x\geq 0)}1_{(t-x\leq 1)}.
 \end{equation}
 The corresponding adjoint solution is given by $\phi(t,x)+\phi(t,1-x)$ where $\phi$ is defined in \eqref{varphi3} 
 with the following initial conditions
$$
(\phi_0,\phi_1)=\big(0,-2x \, 1_{(0,1/2)}(x) +2(x-1) \, 1_{(1/2,1)}(x)\big)\in H^2(\Omega)\times H^1(\Omega).
$$ 
Compatibility conditions are satisfied.

Figure \ref{fig:ex2_bord}-left displays the evolution of $\Vert \chi\partial_x\phi_h(1,\cdot)-v\Vert_{L^2(0,T)}/\Vert v\Vert_{L^2(0,T)}$ w.r.t. $h$ for various pairs of $(p,q)$ and $\chi\equiv 1$. We observe a rate close to $0.75$ for $(p,q)=(1,1)$ and  close to $1.5$ otherwise. Remark that a priori  $u\in H^1(M)$ and $\phi\in H^2(M)$ so that the choice $(p,q)=(2,3)$ does not lead to a better rate than the choice $(p,q)=(1,2)$. Moreover, as expected, the introduction of the cut off $\chi\neq 1$ does not improve here the rate of convergence: see Figure \ref{fig:ex2_bord}-right where similar rates are observed.


\begin{table}[http!]
	\centering
		\begin{tabular}{|l|ccccc|}
			\hline
			 			   $h$  & $1.57\times 10^{-1}$  &      $8.22\times 10^{-2}$  	&  $4.03\times 10^{-2}$ 	&  $2.29\times 10^{-2}$  	&  $1.25\times 10^{-2}$  
						   			\tabularnewline
			\hline
			$\textrm{err}(\phi,\phi_h,1)$ 	  & $7.83\times 10^{-1}$ &      $2.53\times 10^{-1}$ 	& $5.82\times 10^{-2}$	& $1.97\times 10^{-2}$ 	& $5.64\times 10^{-3}$	
			\tabularnewline
			$\textrm{err}(\partial_x\phi,\partial_x\phi_h,1)$ 	  & $1.12\times 10^{0}$ &      $5.03\times 10^{-1}$ 	& $2.02\times 10^{-1}$	& $1.12\times 10^{-1}$ 	& $5.04\times 10^{-2}$	
			\tabularnewline
			$\textrm{err}(u,u_h,1)$ 	  & $1.91\times 10^{-1}$ &      $4.94\times 10^{-2}$ 	& $2.42\times 10^{-2}$	& $1.17\times 10^{-2}$ 	& $5.04\times 10^{-3}$	
			\tabularnewline
			$\textrm{err}(v,u_h)$ 	  & $2.20\times 10^{-1}$ &      $6.48\times 10^{-2}$ 	& $2.87\times 10^{-2}$	& $1.16\times 10^{-2}$  & $4.95\times 10^{-3}$	
			\tabularnewline
			$\textrm{err}(v,h \chi \partial_{x}\phi_h)$ 	  & $2.49\times 10^{-1}$ &      $6.41\times 10^{-2}$ 	& $2.97\times 10^{-2}$	& $1.40\times 10^{-2}$  & $6.06\times 10^{-3}$	
			\tabularnewline
			\hline
		\end{tabular}
		\vspace{0.1cm}
	\caption{\eqref{eq:Ex2} - Boundary case - $(p,q)=(1,2)$; $\chi \equiv 1$ }
	\label{tab:ex2_12}
\end{table}


\begin{table}[http!]
	\centering
		\begin{tabular}{|l|ccccc|}
			\hline
			 			   $h$  & $1.57\times 10^{-1}$  &      $8.22\times 10^{-2}$  	&  $4.03\times 10^{-2}$ 	&  $2.29\times 10^{-2}$  	&  $1.25\times 10^{-2}$  
						   			\tabularnewline
			\hline
			$\textrm{err}(\phi,\phi_h,1)$ 	  & $4.90\times 10^{-1}$ &      $1.35\times 10^{-1}$ 	& $3.28\times 10^{-2}$	& $1.07\times 10^{-2}$ 	& $3.41\times 10^{-3}$	
			\tabularnewline
			$\textrm{err}(\partial_x\phi,\partial_x\phi_h,1)$ 	  & $9.79\times 10^{-1}$ &      $3.75\times 10^{-1}$ 	& $1.35\times 10^{-1}$	& $6.76\times 10^{-2}$ 	& $3.21\times 10^{-2}$	
			\tabularnewline
			$\textrm{err}(u,u_h,1)$ 	  & $6.70\times 10^{-2}$ &      $2.60\times 10^{-2}$ 	& $9.34\times 10^{-3}$	& $3.78\times 10^{-3}$ 	& $1.44\times 10^{-3}$	
			\tabularnewline
			$\textrm{err}(v,u_h)$ 	  & $8.61\times 10^{-2}$ &      $3.25\times 10^{-2}$ 	& $1.10\times 10^{-2}$	& $4.19\times 10^{-3}$  & $1.59\times 10^{-3}$	
			\tabularnewline
			$\textrm{err}(v,h \chi \partial_{x}\phi_h)$ 	  & $8.57\times 10^{-2}$ &      $3.23\times 10^{-2}$ 	& $1.10\times 10^{-2}$	& $4.45\times 10^{-3}$  & $1.71\times 10^{-3}$	
			\tabularnewline
			\hline
		\end{tabular}
		\vspace{0.1cm}
	\caption{\eqref{eq:Ex2} - Boundary case - $(p,q)=(2,3)$; $\chi \equiv 1$ }
	\label{tab:ex2_23}
\end{table}

\begin{figure}[!http]
\begin{center}
\includegraphics[scale=0.45]{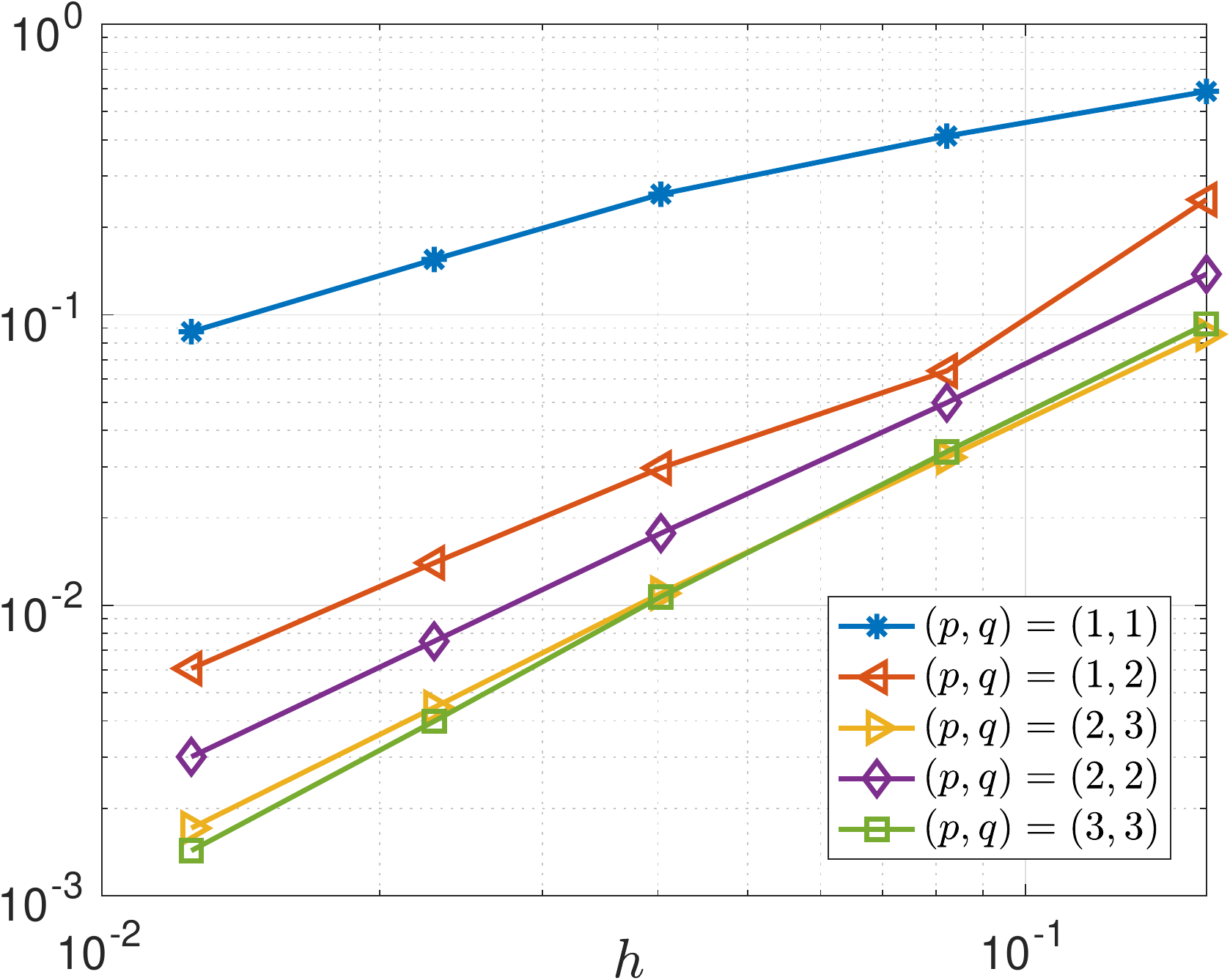}
\includegraphics[scale=0.45]{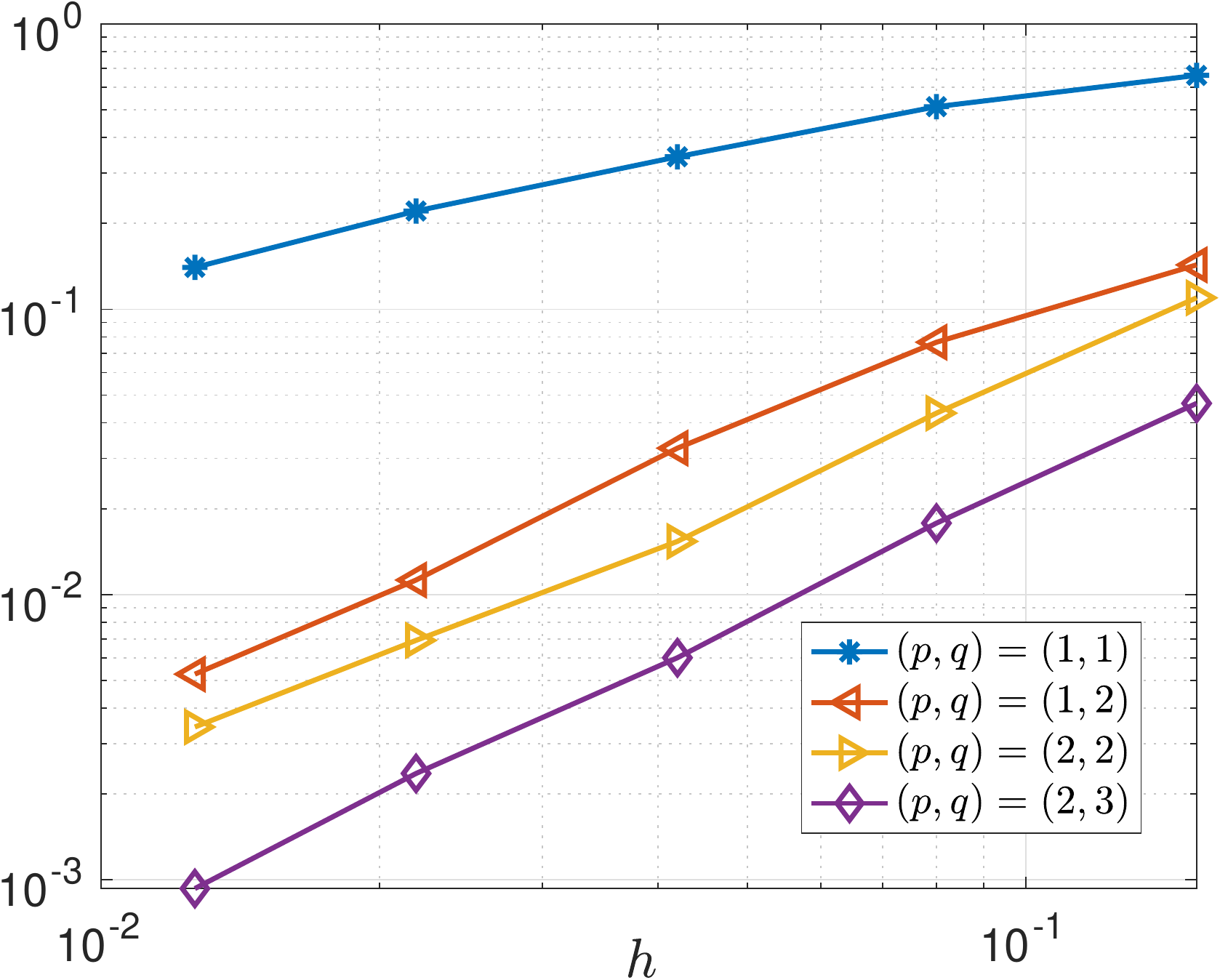}
\caption{\eqref{eq:Ex2} - Boundary case -  $\Vert h\partial_{x}\chi\phi_h(1,\cdot)-v\Vert_{L^2(0,T)}/\Vert v\Vert_{L^2(0,T)}$ vs. $h$ with $\chi=1$ (left) and $\chi=\chi_0$ from \eqref{chi_used} (right).}\label{fig:ex2_bord}
\end{center}
\end{figure}

\subsection{Boundary case: initial condition in $H^{0}(\Omega)\times H^{-1}(\Omega)$}

We consider the following stiff situation given by 
\begin{equation}
\label{eq:Ex3}\tag{\bf Ex3}
(u_0,u_1)=(4x 1_{(0,1/2)}(x),0)\in H^0(\Omega)\times H^{-1}(\Omega).
\end{equation}
and extensively discussed in \cite{cindea_munch_calcolo2015,munch_m2an2005} and $T=2$.
 The corresponding control of minimal $L^2((0,T)\times \{1\})$ norm is given by $v(t)=2(1-t) 1_{(1/2,3/2)}(t)$ leading to $\Vert v\Vert_{L^2(0,T)}=1/\sqrt{3}$. The corresponding controlled solution is explicitly known as follows: 
\begin{equation}
u(t,x)=\left\{
\begin{array}{cc}
 4x & 0\leq x+t <\frac{1}{2}, \\
 2(x-t) &  -\frac{1}{2}< t-x < \frac{1}{2}, \quad x+t\geq \frac{1}{2},\\
 0 & else,
\end{array}
\right.
\end{equation}
leading to $\Vert u\Vert_{L^2(M)}=1/\sqrt{3}$. The corresponding initial conditions of the adjoint solution is $(\phi_0,\phi_1)=(0,-2x \, 1_{(0,1/2)}(x))\in H^1(\Omega)\times H^0(\Omega)$ leading to 
\begin{equation}
\label{varphi3}
\phi(t,x)=\left\{
\begin{array}{cc}
 -2xt & 0\leq x+t <\frac{1}{2}, \quad x\geq 0, t\geq 0,\\
 \frac{(x-t)^2}{2}-\frac{1}{8} &  \frac{1}{2}\leq x+t < \frac{3}{2}, \quad -\frac{1}{2}<x-t< \frac{1}{2},\\
 2(x-1)(1-t) & \frac{3}{2}\leq x+t, \quad -\frac{1}{2}<x-t,\\
  -\frac{(x+t-2)^2}{2}+\frac{1}{8} &  \frac{3}{2}< x+t < \frac{5}{2}, \quad -\frac{3}{2}<x-t\leq  -\frac{1}{2},\\
   2x(2-t) &  x-t\leq -\frac{3}{2},
\end{array}
\right.
\end{equation}
leading to $\Vert \phi\Vert_{L^2(M)}\approx 9.86\times 10^{2}$ and  $\Vert \partial_x\phi\Vert_{L^2(M)}\approx 4.08\times 10^{-1}$. In particular, we check that $\partial_{x} \phi(t,x)_{\vert x=1}=2(1-t)\, 1_{(1/2,3/2)}(t)=v(t).$ Both $u$ and $\phi$ develop singularities (where $u$ and $\nabla \phi$ are discontinuous).

Figure \ref{fig:ex3} depicts the evolution of $\Vert h\chi\partial_{x}\phi_h(\cdot,1)-v\Vert_{L^2(0,T)}/\Vert v\Vert_{L^2(0,T)}$ w.r.t. $h$ with $\chi\equiv1$.  We observe a rate close to $0.5$.

  Let us also emphasize that the space-time discretization formulation is very well appropriated for mesh adaptivity.
   Using the $V_h^1\times V_h^2$ approximation, Figure~\ref{adaptmesh_ex3}-left (resp.~right) depicts the mesh obtained after seven adaptative refinements based on the local values of gradient of the variable $\phi_h$ (resp.~$u_h$).
   Starting with a coarse mesh composed of~$288$ triangles and~$166$ vertices, the final mesh on the right is composed with $13068$~triangles and $6700$~vertices and leads to a relative error $\textrm{err}(v,u_h)$ of the order of $10^{-3}$. The final mesh follows the singularities of the controlled solution starting at the point $(0,1)$ of discontinuity of $u_0$.

\begin{table}[http!]
	\centering
		\begin{tabular}{|l|ccccc|}
			\hline
			 			   $h$  & $1.57\times 10^{-1}$  &      $8.22\times 10^{-2}$  	&  $4.03\times 10^{-2}$ 	&  $2.29\times 10^{-2}$  	&  $1.25\times 10^{-2}$  
						   			\tabularnewline
			\hline
			$\textrm{err}(\phi,\phi_h,1)$ 	  & $6.38\times 10^{-1}$ &      $4.35\times 10^{-1}$ 	& $2.85\times 10^{-1}$	& $1.84\times 10^{-1}$ 	& $1.05\times 10^{-1}$	
			\tabularnewline
			$\textrm{err}(\partial_x\phi,\partial_x\phi_h,1)$ 	  & $8.38\times 10^{-1}$ &      $6.23\times 10^{-1}$ 	& $4.85\times 10^{-1}$	& $3.97\times 10^{-1}$ 	& $3.17\times 10^{-1}$	
			\tabularnewline
			$\textrm{err}(u,u_h,1)$ 	  & $5.78\times 10^{-1}$ &      $4.40\times 10^{-1}$ 	& $3.59\times 10^{-1}$	& $2.93\times 10^{-1}$ 	& $2.27\times 10^{-1}$	
			\tabularnewline
			$\textrm{err}(v,u_h)$ 	  & $7.67\times 10^{-1}$ &      $5.69\times 10^{-1}$ 	& $5.04\times 10^{-1}$	& $4.05\times 10^{-1}$  & $3.15\times 10^{-1}$	
			\tabularnewline
			$\textrm{err}(v,h \chi \partial_x\phi_h)$ 	  & $8.41\times 10^{-1}$ &      $6.47\times 10^{-1}$ 	& $5.08\times 10^{-1}$	& $4.09\times 10^{-1}$  & $3.16\times 10^{-1}$	
			\tabularnewline
			\hline
		\end{tabular}
		\vspace{0.1cm}
	\caption{\eqref{eq:Ex3}; $(p,q)=(1,1)$  - $\chi\equiv 1$ - Boundary case}
	\label{tab:ex3_11}
\end{table}

\begin{table}[http!]
	\centering
		\begin{tabular}{|l|ccccc|}
			\hline
			 			   $h$  & $1.57\times 10^{-1}$  &      $8.22\times 10^{-2}$  	&  $4.03\times 10^{-2}$ 	&  $2.29\times 10^{-2}$  	&  $1.25\times 10^{-2}$  
						   			\tabularnewline
			\hline
			$\textrm{err}(\phi,\phi_h,1)$ 	  & $1.62\times 10^{0}$ &      $6.33\times 10^{-1}$ 	& $2.72\times 10^{-1}$	& $1.45\times 10^{-1}$ 	& $7.36\times 10^{-2}$	
			\tabularnewline
			$\textrm{err}(\partial_x\phi,\partial_x\phi_h,1)$ 	  & $2.33\times 10^{0}$ &      $1.52\times 10^{0}$ 	& $1.22\times 10^{0}$	& $1.08\times 10^{0}$ 	& $1.05\times 10^{0}$	
			\tabularnewline
			$\textrm{err}(u,u_h,1)$	  & $3.93\times 10^{-1}$ &      $3.00\times 10^{-1}$ 	& $2.27\times 10^{-1}$	& $1.74\times 10^{-1}$ 	& $1.30\times 10^{-1}$	
			\tabularnewline
			$\textrm{err}(v,u_h)$ 	  & $5.03\times 10^{-1}$ &      $3.43\times 10^{-1}$ 	& $2.41\times 10^{-1}$	& $1.89\times 10^{-1}$  & $1.48\times 10^{-1}$	
			\tabularnewline
			$\textrm{err}(v,h \chi \partial_x\phi_h)$ 	  & $4.73\times 10^{-1}$ &      $3.41\times 10^{-1}$ 	& $2.64\times 10^{-1}$	& $2.08\times 10^{-1}$  & $1.60\times 10^{-1}$	
			\tabularnewline
			\hline
		\end{tabular}
		\vspace{0.1cm}
	\caption{\eqref{eq:Ex3}; $(p,q)=(1,2)$ - $\chi\equiv 1$ - Boundary case}
	\label{tab:ex3_12}
\end{table}

\begin{table}[http!]
	\centering
		\begin{tabular}{|l|ccccc|}
			\hline
			 			   $h$  & $1.57\times 10^{-1}$  &      $8.22\times 10^{-2}$  	&  $4.03\times 10^{-2}$ 	&  $2.29\times 10^{-2}$  	&  $1.25\times 10^{-2}$  
						   			\tabularnewline
			\hline
			$\textrm{err}(\phi,\phi_h,1)$ 	  & $1.25\times 10^{0}$ &      $5.01\times 10^{-1}$ 	& $2.02\times 10^{-1}$	& $9.58\times 10^{-2}$ 	& $4.52\times 10^{-2}$	
			\tabularnewline
			$\textrm{err}(\partial_x\phi,\partial_x\phi_h,1)$ 	  & $3.44\times 10^{0}$ &      $2.34\times 10^{0}$ 	& $1.88\times 10^{0}$	& $1.57\times 10^{0}$ 	& $1.45\times 10^{0}$	
			\tabularnewline
			$\textrm{err}(u,u_h,1)$	  & $3.12\times 10^{-1}$ &      $2.22\times 10^{-1}$ 	& $1.60\times 10^{-1}$	& $1.24\times 10^{-1}$ 	& $9.01\times 10^{-2}$	
			\tabularnewline
			$\textrm{err}(v,u_h)$ 	  & $3.52\times 10^{-1}$ &      $2.36\times 10^{-1}$ 	& $1.71\times 10^{-1}$	& $1.30\times 10^{-1}$  & $9.64\times 10^{-2}$	
			\tabularnewline
			$\textrm{err}(v,h \chi \partial_x\phi_h)$ 	  & $3.40\times 10^{-1}$ &      $2.36\times 10^{-1}$ 	& $1.66\times 10^{-1}$	& $1.37\times 10^{-1}$  & $9.56\times 10^{-2}$	
			\tabularnewline
			\hline
		\end{tabular}
		\vspace{0.1cm}
	\caption{\eqref{eq:Ex3}; $(p,q)=(2,3)$ - $\chi\equiv1$ - Boundary case}
	\label{tab:ex3_23}
\end{table}

\begin{figure}[!http]
\begin{center}
\includegraphics[scale=0.55]{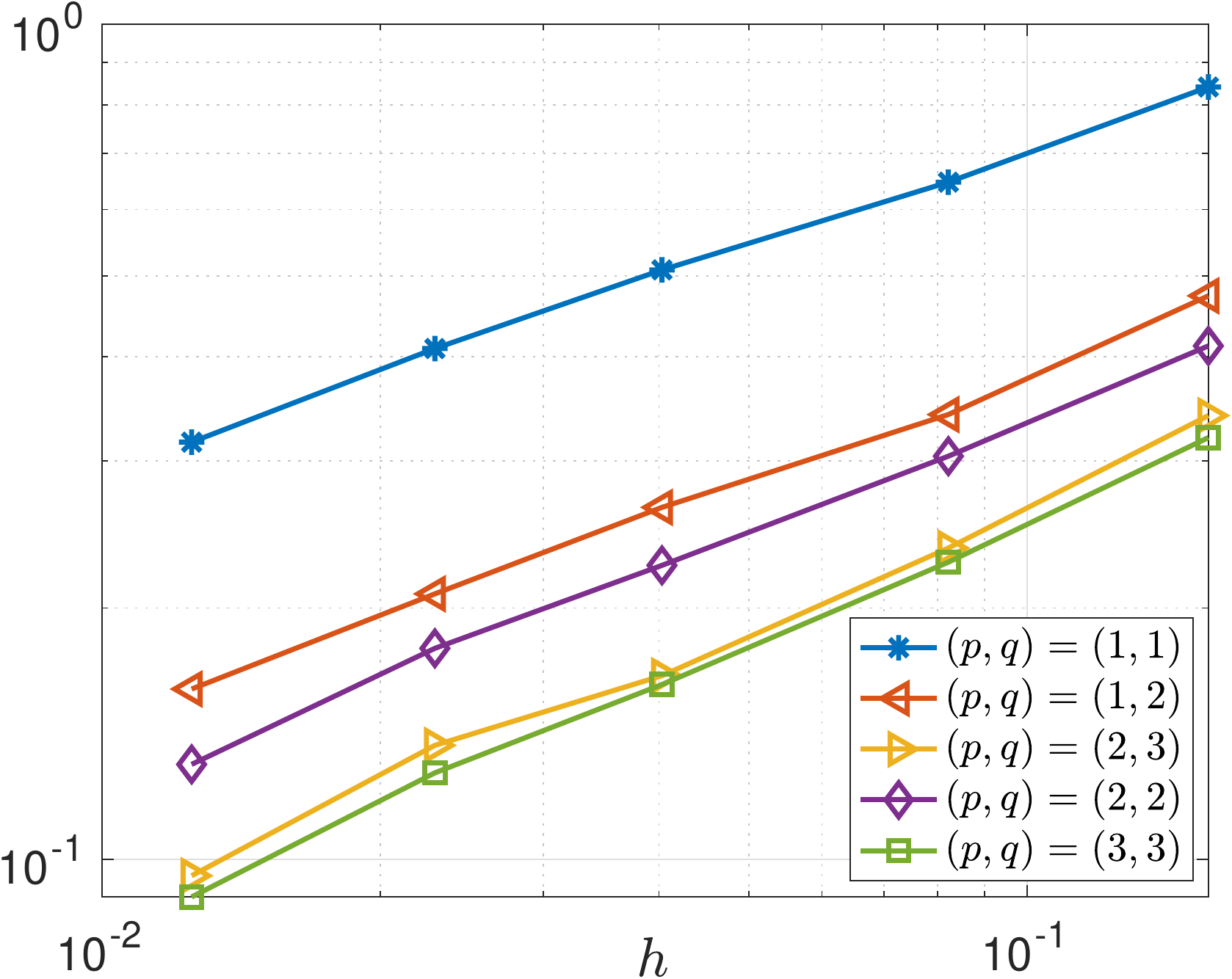}
\caption{\eqref{eq:Ex3}; $\Vert h\chi\partial_{x}\phi_h(\cdot,1)-v\Vert_{L^2(0,T)}/\Vert v\Vert_{L^2(0,T)}$ w.r.t. $h$ (rate $\approx 0.5$).}\label{fig:ex3}
\end{center}
\end{figure}

\begin{figure}[!http]
\begin{center}
\includegraphics[scale=0.7]{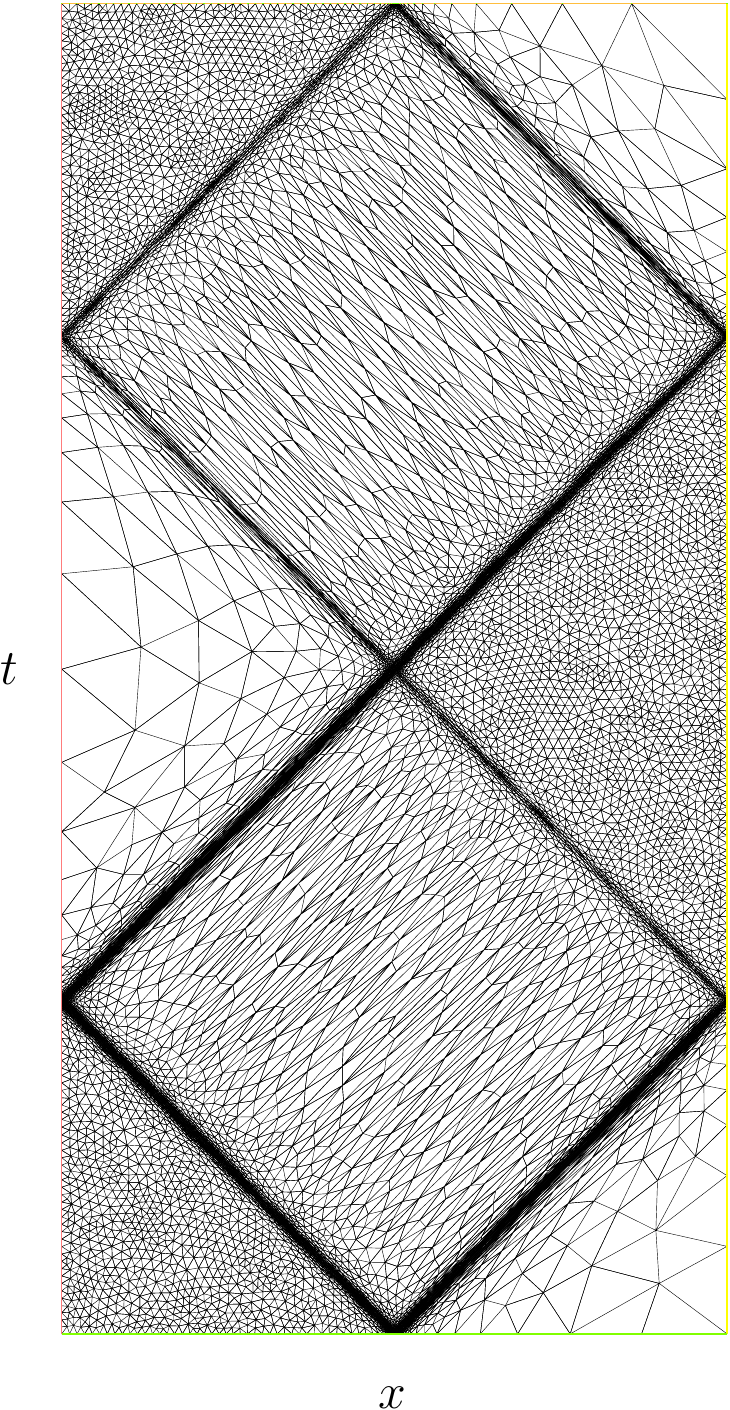}\hspace*{1.5cm}
\includegraphics[scale=0.7]{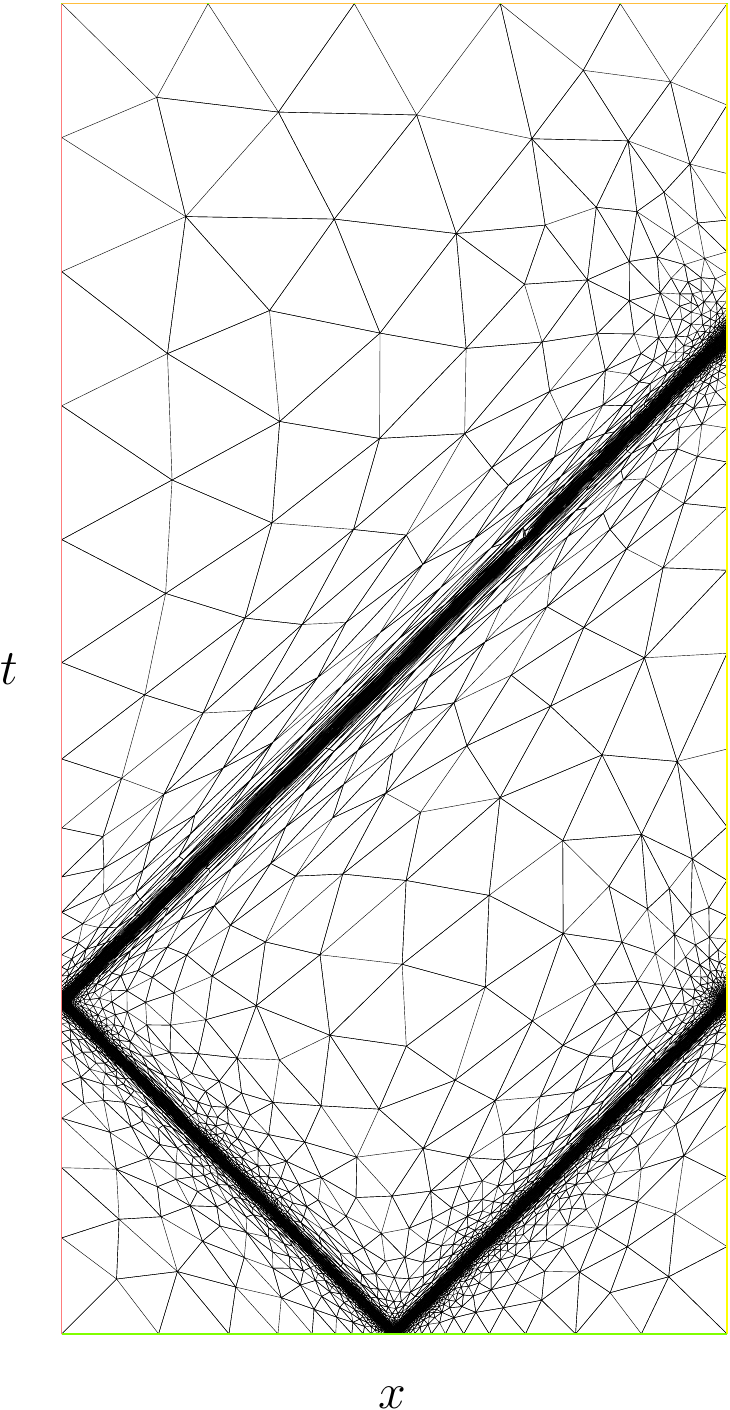}
\caption{\eqref{eq:Ex3};   Locally refine space-time meshes with respect to $\phi_h$ (left) and $u_h$ (right). $(p,q)=(1,2)$}\label{adaptmesh_ex3}
\end{center}
\end{figure}

\subsection{Boundary case: the wave equation with a potential}
To end these numerical illustrations, we report some results for the wave equation with non vanishing potential $V$, see \eqref{def_P}. Non zero potentials notably appear from linearization of nonlinear wave equations of the form $\Box u+ f(u)=\chi v$ (see \cite{munch2020constructive}). Actually, we want to emphasize that this spacetime approach, based on the resolution of the optimal condition associated with the control of minimal $L^2$ norm is very relevant for potential with the ``bad'' sign for which $V(t,x)u(t,x)<0$.
Indeed, in this case, the usual ``\`a la Glowinski'' strategy developed in \cite{Glowinski_Japan_1990} is numerically inefficient and requires adaptations, since the uncontrolled solution (used to initialize the conjugate algorithm) grows exponentially in time, leading to numerical instabilities and overflow. Recall that the observability constant behaves like  
$e^{C(T,\omega)\Vert A\Vert^2_{L^\infty(0,T; L^n(\Omega))}}$ (see \cite{XuZhang_2000}) and appears notably in the constant in the a priori estimate \eqref{est_control_bd}. We consider the initial condition \eqref{eq:Ex1}, $T=2.5$ and constant negative potentials $V(t,x)=V<0$. Table~\ref{tab:ex1_12_V} collects the relative error on the approximation of the boundary control with respect to $h$ for several negatives values of $V$. In particular, for $V=-40$, the $L^\infty$ norm of the corresponding uncontrolled solution is of order $10^5$. We approximate $u$ and $\phi$ in $V_h^1$ and $V_h^2$ respectively and observe a rate close to $1.5$. The value of $V$ only affects the constant. We refer to \cite{Micu_2020} for a semi-discrete (in space) approximation of exact boundary controls for a semi-discretized wave equation with potential, including experiments for small and potentials with good sign. 

\begin{table}[http!]
	\centering
		\begin{tabular}{|l|ccccc|}
			\hline
			 			   $h$  & $1.6\times 10^{-1}$  &      $8\times 10^{-2}$  	&  $4\times 10^{-2}$ 	&  $2\times 10^{-2}$  	&  $1\times 10^{-2}$  
						   			\tabularnewline
			\hline
			$V=-10$ 	  & $7.11\times 10^{-1}$ &      $2.43\times 10^{-1}$ 	& $6.91\times 10^{-2}$	& $2.48\times 10^{-2}$ 		 		& $7.97\times 10^{-3}$	
			\tabularnewline
			$V=-20$  	  & $7.06\times 10^{-1}$ &      $2.81\times 10^{-1}$ 	& $1.57\times 10^{-1}$	& $6.13\times 10^{-2}$ 		 		& $1.65\times 10^{-2}$	
			\tabularnewline
			$V=-30$ 	  & $9.54\times 10^{-1}$ &      $6.91\times 10^{-1}$ 	& $2.07\times 10^{-1}$	& $7.21\times 10^{-2}$ 		 		& $2.27\times 10^{-2}$	
			\tabularnewline
			$V=-40$ 	  & $1.01\times 10^{-1}$ &      $9.44\times 10^{-1}$ 	& $5.05\times 10^{-1}$	& $1.13\times 10^{-1}$ 		 		& $3.18\times 10^{-2}$	
			\tabularnewline
			\hline
		\end{tabular}
		\vspace{0.1cm}
	\caption{\eqref{eq:Ex1} -  Boundary case -  $(p,q)=(1,2)$ - $\chi$ from \eqref{chi_used};  $\Vert h\chi\partial_{x}\phi_h(\cdot,1)-v\Vert_{L^2(0,T)}/\Vert v\Vert_{L^2(0,T)}$  w.r.t. $h$ and $V\in \{-10,-20,-30,-40\}$. }
	\label{tab:ex1_12_V}
\end{table}

\section{Conclusion}

We have introduced and analyzed  a spacetime finite element approximation of a controllability problem for the wave equation. Based on a non conformal $H^1$-approximation, the analysis yields error estimates for the control in the natural $L^2$ -norm of order $h^q$ (resp. $h^{q-\frac{1}{2}}$)  where $q$ is the degree of the polynomials used to describe  the adjoint variable in the distributed (resp. boundary) case. The numerical experiments performed for initial data with various regularity exhibits the efficiency method. The convergence is also observed for initial data with minimal regularity.  

We emphasize that spacetime formulations are easier to implement than time-marching methods, since in particular, there is no kind of CFL condition between the time and space discretization parameters. Moreover, as shown in the numerical section, they are well-suited for mesh adaptivity (as initially discussed in \cite{hughes_spacetime_1990}).

Similarly to the formulation proposed in \cite{cindea_efc_munch_2013, cindea_munch_calcolo2015}, the present formulation follows the ``control then discretize'' approach. However, contrary to \cite{cindea_efc_munch_2013, cindea_munch_calcolo2015}, the $H^1$-formulation of the present work does not require the introduction of sophisticated finite element spaces. On the other hand, the formulation requires additional stabilized terms which are function of the jump of the gradient across the boundary of each element. The analysis is then inspired from  \cite{BFO_Sicon2020}  and also from \cite{BFMO_2021}
where an analogous spacetime formulation for a data assimilation is considered. 

The implementation of the stabilized terms is not straightforward, in particular, in higher dimension, and is usually not available in finite element softwares. A possible way to circumvent the introduction of the gradient jump terms is to consider non-conforming approximation of the Crouzeix-Raviart type as in \cite{Bur17}. A penalty is then needed on the solution jump instead to control the $H^1$-conformity error.
Another possible way, following \cite{montaner_munch_2019} devoted to the boundary case, could be to consider the controllability problem associated to a first order reformulation of the wave equation: 
\begin{equation}\label{Intro:first_order_formulation}
\begin{cases}
v_t-div\, \pp =0,\\
\pp_t-\nabla v =0.
\end{cases}
\end{equation}
with  $v:=u_t$ and $\pp:=\nabla u$. A $H^1$ conformal stabilized
approximation is employed in \cite{montaner_munch_2019} leading to
promising numerical experiments in the one dimensional case. A
rigorous numerical analysis however remains to be done.

\appendix

\appendix

\section{Continuum estimates}
\label{appendix_cont}

\begin{proposition}[Energy estimate]
\label{prop_energy}
There holds
    \begin{align}\label{direct0}
&\norm{u}_{L^\infty(0,T; L^2(\Omega))}
+
\norm{\p_t u}_{L^2(0,T; H^{-1}(\Omega))}
+
\norm{\p_\nu u}_{H^{-1}((0,T) \times \p \Omega)}
\\\notag&\quad\lesssim 
\norm{u|_{t=0}}_{L^2(\Omega)}
+
\norm{\p_t u|_{t=0}}_{H^{-1}(\Omega)}
+
\norm{u}_{L^2((0,T) \times \p \Omega)}
+
\norm{\Box u}_{H^{-1}((0,T) \times \Omega)}.
    \end{align}
\end{proposition}
\begin{proof}
The estimate 
    \begin{align}\label{LLT_est}
&\norm{u}_{L^\infty(0,T; L^2(\Omega))}
+
\norm{\p_t u}_{L^\infty(0,T; H^{-1}(\Omega))}
+
\norm{\p_\nu u}_{H^{-1}((0,T) \times \p \Omega)}
\\\notag&\quad\lesssim 
\norm{u|_{t=0}}_{L^2(\Omega)}
+
\norm{\p_t u|_{t=0}}_{H^{-1}(\Omega)}
+
\norm{u}_{L^2((0,T) \times \p \Omega)}
+
\norm{\Box u}_{L^1(0,T; H^{-1}(\Omega))}
    \end{align}
follows from \cite[Theorem 2.3]{LLT}, see also Remark 2.2 there. Thus 
it is enough to consider the equation
    \begin{align}\label{wave_with_rhs}
\begin{cases}
\Box u = f,
\\
u|_{x \in \p \Omega} = 0,
\\
u|_{t < 0} = 0,
\end{cases}
    \end{align}
and show that its solution satisfies
    \begin{align*}
&\norm{u}_{L^\infty(0,T; L^2(\Omega))}
+
\norm{\p_t u}_{L^2(0,T; H^{-1}(\Omega))}
+
\norm{\p_\nu u}_{H^{-1}((0,T) \times \p \Omega)}
\lesssim
\norm{f}_{H^{-1}((0,T) \times \Omega)}.
    \end{align*}
Let us use the shorthand notations $M = (0,T) \times \Omega$
and $(t,x) = (x^0, \dots, x^n)$.
We recall that for any $f \in H^{-1}(M)$
there are $f_j \in L^2(M)$, $j=-1,0,\dots,n$,
such that 
    \begin{align}\label{Hm1_exp}
f = f_{-1} + \sum_{j=0}^n \p_{x^j} f_j,
    \end{align}
and that 
    \begin{align*}
\norm{f}_{H^{-1}(M)}^2 = \inf \sum_{j=-1}^n \norm{f_j}_{L^2(M)}^2,
    \end{align*}
where the infimum is taken over all $f_j \in L^2(M)$ satisfying (\ref{Hm1_exp}), see e.g. \cite[Theorem 1, p. 299]{EvansPDEbook}.
Thus it is enough to show that 
    \begin{align*}
&\norm{u}_{L^\infty(0,T; L^2(\Omega))}
+
\norm{\p_t u}_{L^2(0,T; L^2(\Omega))}
+
\norm{\p_\nu u}_{H^{-1}((0,T) \times \p \Omega)}
\lesssim
\norm{f_j}_{L^2((0,T) \times \Omega)},
    \end{align*}
where $u$ satisfies (\ref{wave_with_rhs}) 
with $f$ replaced by $\p_{x^j} f_j$ when $j \ge 0$
and by $f_{-1}$ when $j=-1$. 
The cases $j=-1$ and $j > 0$ are contained in (\ref{LLT_est}). 

Let us consider the case $j=0$.
We denote by $v$ the solution of (\ref{wave_with_rhs})
with $f = f_0$.
Then $u = \p_t v$ and it follows from \cite[Theorem 2.1]{LLT} that
    \begin{align*}
\norm{u}_{L^\infty(0,T; L^2(\Omega))}
+
\norm{\p_\nu u}_{H^{-1}((0,T) \times \p \Omega)}
&\lesssim
\norm{\p_t v}_{L^\infty(0,T; L^2(\Omega))}
+
\norm{\p_\nu v}_{L^2((0,T) \times \p \Omega)}
\\&\lesssim
\norm{f_0}_{L^2((0,T) \times \Omega)}.
    \end{align*}
Moreover, 
    \begin{align*}
\norm{\p_t u}_{L^2(0,T; H^{-1}(\Omega))}
=
\norm{\p_t^2 v}_{L^2(0,T; H^{-1}(\Omega))}
\le
\norm{\Delta v}_{L^2(0,T; H^{-1}(\Omega))}
+
\norm{f_0}_{L^2(0,T; H^{-1}(\Omega))},
    \end{align*}
and using \cite[Theorem 2.1]{LLT} again,
    \begin{align*}
\norm{\Delta v}_{L^2(0,T; H^{-1}(\Omega))}
\lesssim
\norm{v}_{L^2(0,T; H^1(\Omega))}
\lesssim
\norm{f_0}_{L^2((0,T) \times \Omega)}.
    \end{align*}
\end{proof}

\begin{remark}\label{rem_opt_rough}
It is not possible to improve the $L^2(0,T; H^{-1}(\Omega))$ norm of $\p_t u$ on the left-hand side of (\ref{direct0})
to its $L^\infty(0,T; H^{-1}(\Omega))$ norm.
\end{remark}
\begin{proof}
To get a contradiction, we suppose that 
    \begin{align*}
\norm{\p_t u}_{L^\infty(0,T; H^{-1}(\Omega))}
\lesssim 
\norm{\Box u}_{H^{-1}((0,T) \times \Omega)}
    \end{align*}
for solutions $u$ of (\ref{wave_with_rhs}).
Let $f \in L^2(0,T)$, and denote by $u$ and $v$ the solutions of (\ref{wave_with_rhs}) with sources $\p_t f$ and $f$, respectively.  
Then $u = \p_t v$ and
    \begin{align*}
f = \p_t^2 v - \Delta v = \p_t u - \Delta v \in L^\infty(0,T; H^{-1}(\Omega)).
    \end{align*}
But this implies that $f \in L^\infty(0,T)$.
As $f\in L^2(0,T)$ was arbitrary, we get the contradiction 
$L^2(0,T) \subset L^\infty(0,T)$.
\end{proof}

If $\Box u = 0$ and $u|_{x \in \p \Omega} = 0$ then the norm on the left-hand side of (\ref{direct0})
controls the $L^\infty(0,T; H^{-1}(\Omega))$ norm of $\p_t u$.
In fact, we have:

\begin{lemma}
\label{lem_linfty_ptu}
Suppose that $u \in L^2((0,T) \times \Omega)$ satisfies
$\Box u = 0$ and $u|_{x \in \p \Omega} = 0$. Then 
    \begin{align*}
&\norm{u}_{L^\infty(0,T; L^2(\Omega))}
+
\norm{\p_t u}_{L^\infty(0,T; H^{-1}(\Omega))}
+ 
\norm{\p_\nu u}_{H^{-1}((0,T) \times \p \Omega)}
\lesssim 
\norm{u}_{L^2((0,T) \times \Omega)}.
    \end{align*}
Moreover, $u \in C(0,T; L^2(\Omega)) \cap  C^1(0,T; H^{-1}(\Omega))$.
\end{lemma} 
\begin{proof}
Let $\chi \in C^\infty(\R)$ satisfy $\chi(t) = 0$ near $t=T$
and $\chi(t) = 1$ for $t \in (0,T/2)$.
Applying Proposition \ref{prop_energy} to $\chi u$ backwards in time, we obtain
    \begin{align*}
&\norm{u}_{L^\infty(0,T/2; L^2(\Omega))}
+
\norm{\p_t u}_{L^2(0,T/2; H^{-1}(\Omega))}
\lesssim \norm{[\Box, \chi] u}_{H^1((0,T) \times \Omega)}
\lesssim \norm{u}_{L^2((0,T) \times \Omega)}.
    \end{align*}
Applying (\ref{LLT_est}) backwards in time on the interval $(0,s)$
where $s < T/2$, we get
    \begin{align*}
\norm{\p_t u|_{t=0}}_{H^{-1}(\Omega)}^2
\lesssim 
\norm{u|_{t=s}}_{L^2(\Omega)}^2
+
\norm{\p_t u|_{t=s}}_{H^{-1}(\Omega)}^2.
    \end{align*}
Integration in $s$ gives
    \begin{align*}
\norm{\p_t u|_{t=0}}_{H^{-1}(\Omega)}^2
\lesssim 
&\norm{u}_{L^2(0,T/2; L^2(\Omega))}^2
+
\norm{\p_t u}_{L^2(0,T/2; H^{-1}(\Omega))}^2.
    \end{align*}
We conclude that 
    \begin{align*}
\norm{u|_{t=0}}_{L^2(\Omega)}
+
\norm{\p_t u|_{t=0}}_{H^{-1}(\Omega)}
\lesssim
\norm{u}_{L^2((0,T) \times \Omega)}.
    \end{align*}
The claimed estimate follows from (\ref{LLT_est}).

Let us now turn the claimed continuity.
Let $\epsilon > 0$ and $u_j$ be a mollification in time 
that satisfies $u_j \to u$ in $L^2((\epsilon,T-\epsilon) \times \Omega)$. Then $u_j$ converges in 
$$C(\epsilon,T-\epsilon; L^2(\Omega)) \cap  C^1(\epsilon,T-\epsilon; H^{-1}(\Omega))$$
and thus $u$ is in this space.
In particular, $u|_{t=T/2}$ and $\p_t u|_{t=T/2}$ are well-defined.
Solving the initial value problem starting from these gives the desired conclusion.
\end{proof} 

\begin{theorem}[Distributed observability estimate]
\label{th_obs_dist}
Let $T>0$ and let an open set $\omega \subset \Omega$ 
satisfy the geometric control condition.
Then
    \begin{align*}
&\norm{\phi|_{t=0}}_{L^2(\Omega)}
+
\norm{\p_t \phi|_{t=0}}_{H^{-1}(\Omega)}
\lesssim
\norm{\phi}_{L^2((0,T) \times \omega)}
+ 
\norm{\phi}_{L^2((0,T) \times \p \Omega)}
+
\norm{\Box \phi}_{H^{-1}((0,T) \times \Omega)}.
    \end{align*}
\end{theorem}
\begin{proof}
It is classical that 
    \begin{align}\label{obs_dist_classical}
&\norm{\psi|_{t=0}}_{L^2(\Omega)}
+
\norm{\p_t \psi|_{t=0}}_{H^{-1}(\Omega)}
\lesssim
\norm{\psi}_{L^2((0,T) \times \omega)}
    \end{align}
when 
    \begin{align*}
&\begin{cases}
\Box \psi = 0,
\\
\psi|_{x \in \p \Omega} =0,
\end{cases}
    \end{align*}
see e.g. \cite{LLTT17}.
Let $u$ solve 
    \begin{align*}
&\begin{cases}
\Box u = \Box \phi,
\\
u|_{x \in \p \Omega} = \phi|_{x \in \p \Omega},
\\
u|_{t = 0} = 0,\ \p_t u|_{t = 0} = 0,
\end{cases}
    \end{align*}
and define $\psi = \phi - u$.
Then using (\ref{obs_dist_classical}) and (\ref{direct0}),
    \begin{align*}
&\norm{\phi|_{t=0}}_{L^2(\Omega)}
+
\norm{\p_t \phi|_{t=0}}_{H^{-1}(\Omega)}
=\norm{\psi|_{t=0}}_{L^2(\Omega)}
+
\norm{\p_t \psi|_{t=0}}_{H^{-1}(\Omega)}
\\&\quad\lesssim
\norm{\psi}_{L^2((0,T) \times \omega)}
\le
\norm{\phi}_{L^2((0,T) \times \omega)}
+ 
\norm{u}_{L^2((0,T) \times \omega)}
\\&\quad\lesssim
\norm{\phi}_{L^2((0,T) \times \omega)}
+
\norm{\phi}_{L^2((0,T) \times \p \Omega)}
+
\norm{\Box \phi}_{H^{-1}((0,T) \times \Omega)}.
    \end{align*}
\end{proof}

\begin{remark}\label{rem_obs_variants}
By applying Theorem \ref{th_obs_dist} to the function $(t,x) \mapsto \phi(T-t,x)$ we obtain the following variant
    \begin{align*}
&\norm{\phi|_{t=T}}_{L^2(\Omega)}
+
\norm{\p_t \phi|_{t=T}}_{H^{-1}(\Omega)}
\lesssim
\norm{\phi}_{L^2((0,T) \times \omega)}
+ 
\norm{\phi}_{L^2((0,T) \times \p \Omega)}
+
\norm{\Box \phi}_{H^{-1}((0,T) \times \Omega)},
    \end{align*}
and by combining Proposition \ref{prop_energy} and Theorem \ref{th_obs_dist}, we get
    \begin{align*}
&\norm{\phi}_{L^\infty(0,T; L^2(\Omega))}
+
\norm{\p_t \phi}_{L^2(0,T; H^{-1}(\Omega))}
\\&\quad\lesssim
\norm{\phi}_{L^2((0,T) \times \omega)}
+ 
\norm{\phi}_{L^2((0,T) \times \p \Omega)}
+
\norm{\Box \phi}_{H^{-1}((0,T) \times \Omega)},
    \end{align*}
assuming that $T>0$ and $\omega \subset \Omega$ 
satisfy the geometric control condition.
\end{remark}

\begin{theorem}[Boundary observability estimate]
\label{th_obs_bd}
Let $T>0$ and let an open set $\omega \subset \p\Omega$ 
satisfy the geometric control condition.
Let $V \in C^\infty(\Omega)$ and define $P$ by \eqref{def_P}.
Then
    \begin{align}\label{obs_bd}
&\norm{\phi|_{t=0}}_{L^2(\Omega)}
+
\norm{\p_t \phi|_{t=0}}_{H^{-1}(\Omega)}
\\\notag&\quad\lesssim
\norm{\p_\nu \phi}_{H^{-1}((0,T) \times \omega)}
+ 
\norm{\phi}_{L^2((0,T) \times \p \Omega)}
+
\norm{P \phi}_{H^{-1}((0,T) \times \Omega)}.
    \end{align}
\end{theorem}
\begin{proof}
It is well-known that 
    \begin{align}\label{obs}
\norm{\psi|_{t=0}}_{L^2(\Omega)}
+
\norm{\p_t \psi|_{t=0}}_{H^{-1}(\Omega)}
\lesssim
\norm{\p_\nu \psi}_{H^{-1}((0,T) \times \omega)}
    \end{align}
for solutions $\psi$ of
    \begin{align}\label{homog}
\begin{cases}
P \psi = 0, 
\\ 
\psi|_{x \in \p \Omega} = 0.
\end{cases}
    \end{align}
However, we did not find this exact formulation in the literature, and give a short proof for the convenience of the reader.
\HOX{To Arnaud: do you know a good ref?}
It follows from (3.11) of the classical paper \cite{BLRII}
that
    \begin{align}\label{BLR_est}
\norm{\psi}_{L^2((0, T) \times \Omega)}
\lesssim \norm{\p_\nu \psi}_{H^{-1}((0,T) \times \omega)},
    \end{align}
since the space of invisible solutions is empty in our case due to unique continuation.     
Let $\tau \in C^\infty(\R)$
satisfy $\tau(t) = 1$ near $t = 0$ and $\tau(t) = 0$ near $t=T$.
Writing $u = \tau \psi$ and $f = \p_t \tau\, \p_t \psi + \p_t^2 \tau\, \psi$, there holds
    \begin{align*}
\begin{cases}
P u = f,
\\
u|_{x \in \p \Omega} = 0,
\\
u|_{t > T} = 0.
\end{cases}
    \end{align*}
As $\psi \in L^2((0,T) \times \Omega)$,
we have using $P \psi = 0$,
    \begin{align*}
\norm{\p_t^2 \psi}_{L^2(0,T; H^{-2}(\Omega))}
=
\norm{\Delta \psi}_{L^2(0,T; H^{-2}(\Omega))}
+ \norm{V \psi}_{L^2(0,T; H^{-2}(\Omega))}
\lesssim
\norm{\psi}_{L^2((0,T) \times \Omega)}.
    \end{align*}
Now interpolation, see e.g. \cite[Theorems 2.3 and 12.2]{Lions1972},  gives
    \begin{align*}
\norm{\p_t \psi}_{L^2(0,T; H^{-1}(\Omega))}
\lesssim 
\norm{\psi}_{L^2(0,T; L^2(\Omega))}.
    \end{align*}
Hence also
    \begin{align*}
\norm{f}_{L^2(0,T; H^{-1}(\Omega))}
\lesssim 
\norm{\psi}_{L^2(0,T; L^2(\Omega))},
    \end{align*}
and (\ref{LLT_est}), or rather its analogue backward in time, gives 
    \begin{align*}
\norm{u|_{t=0}}_{L^2(\Omega)} 
+ 
\norm{\p_t u|_{t=0}}_{H^{-1}(\Omega)} 
\lesssim 
\norm{\psi}_{L^2(0,T; L^2(\Omega))}.
    \end{align*}
But the state of $u$ at $t=0$ coincides with that of $\psi$,
and (\ref{obs}) follows from the above estimate and (\ref{BLR_est}).

Will now show (\ref{obs_bd}).
Let $v$ be the solution of
    \begin{align*}
\begin{cases}
P v = P \phi
\\
v|_{x \in \p \Omega} = \phi|_{x \in \p \Omega}
\\
v|_{t < 0} = 0.
\end{cases}
    \end{align*}
Then $\psi = \phi - v$ solves (\ref{homog}),
and (\ref{obs}) and (\ref{direct0}) imply
    \begin{align*}
&\norm{\phi|_{t=0}}_{L^2(\Omega)}
+
\norm{\p_t \phi|_{t=0}}_{H^{-1}(\Omega)}
= 
\norm{\psi|_{t=0}}_{L^2(\Omega)}
+
\norm{\p_t \psi|_{t=0}}_{H^{-1}(\Omega)}
\\&\quad\lesssim
\norm{\p_\nu \phi}_{H^{-1}((0,T) \times \omega)}
+
\norm{\p_\nu v}_{H^{-1}((0,T) \times \omega)}
\\&\quad\lesssim
\norm{\p_\nu \phi}_{H^{-1}((0,T) \times \omega)}
+
\norm{\phi}_{L^2((0,T) \times \p \Omega)}
+
\norm{P \phi}_{H^{-1}((0,T) \times \Omega)}.
    \end{align*}
\end{proof}

The analogue of Remark \ref{rem_obs_variants} holds also in the case of boundary observations. 
We need also the following classical result. 

\begin{lemma}[Partial hypoellipticity]\label{lem_parthypo}
Let $I \subset (0,T)$ be a compact interval. Suppose that $s \in \R$
and $j \in \N$ satisfy $s+1/2 > j$. 
Then
    \begin{align*}
\norm{\p_\nu^j u}_{H^{s-j-\frac32}(I \times \p \Omega)}
\lesssim \norm{u}_{H^s(M)} + \norm{\Box u}_{H^{s-1}(M)}.
    \end{align*}
\end{lemma} 
\begin{proof}
We define a norm $\norm{u}_X$ by the right-hand side of the claimed inequality, and set $X = \{u \in H^s(M) : \norm{u}_X < \infty \}$.
It follows from the closed graph theorem that $X$ is a Banach space. 
In normal coordinates of $\R \times \p \Omega$, there holds
    \begin{align*}
\Box = \p_\nu^2 + A,
    \end{align*}
where $A$ is a differential operator in the tangential directions to $\R \times \p \Omega$, with coefficients depending on all the variables, see e.g. \cite[Corollary C.5.3]{H3}.

We will use the spaces $\bar H_{(m,s)}(I \times \Omega)$, defined on p. 478 of \cite{H3},
in the boundary normal coordinates, and use the shorthand notation $H_{(m,s)}$ for them. Here $m$ measures Sobolev smoothness in all the variables and $s$ additional smoothness in the tangential variables. However, $s$ can be also negative, corresponding to a loss of smoothness in tangential directions. 

Let $u \in X$. It follows from \cite[Theorem B.2.9]{H3} that
$u \in H_{(m,r)}$ when $m+r\le s-1$ and $m \le s + 1$.
In particular, $u \in H_{(s+1,-2)}$ and the closed graph theorem implies
    \begin{align*}
\norm{u}_{H_{(s+1,-2)}} \lesssim \norm{u}_X.
    \end{align*}
Moreover, using the assumption $s+1/2 > j$, \cite[Theorem B.2.7]{H3} implies 
    \begin{align*}
\norm{\p_\nu^j u}_{H^{s-j-\frac32}(I \times \p \Omega)}
\lesssim \norm{u}_{H_{(s+1,-2)}}.
    \end{align*}
\end{proof} 

\section{Estimates for meshes fitted to the boundary}
\label{appendix_bd}

\begin{proof}[Proof of Lemma \ref{lem_trace_bd}]
Let $u \in C^\infty(K)$.
Let $h>0$, $K \in \T$ and 
consider spherical coordinates $(r,\theta) \in (0,\infty) \times S^n$ centered at $x$
where $x$ is as in (T). It follows from (\ref{T_reg}) that $K$
is star-shaped with respect to $x$. In particular, there is $R : S^n \to (0,\infty)$ such that 
    \begin{align*}
 K = \{(r,\theta) : 0 \le r \le R(\theta),\ \theta \in S^n\}.
    \end{align*}
As $\p K$ is piecewise smooth, it follows from (\ref{T_reg}) that $R$ is piecewise smooth. Applying \cite[Theorem 6, p. 713]{EvansPDEbook} in a piecewise manner, we see that 
    \begin{align*}
\int_{\p K} u^2\,\nu \cdot \rho \d s = \int_{S^n} u^2(R(\theta), \theta) R(\theta)^n \d\theta,
    \end{align*}
where $\mbox{d}\theta$ is the canonical volume measure on the unit sphere $S^n$. It follows from (\ref{T_unif}) that $h \lesssim R(\theta)$. Hence, using (\ref{T_reg}) we have
    \begin{align*}
h \int_{\p K} u^2 \d s 
&\lesssim 
\int_{S^n} u^2(R(\theta), \theta) R(\theta)^{1+n} \d\theta
= 
\int_{S^n} \int_0^{R(\theta)} \p_r (u^2(r, \theta) r^{1+n}) \d r \d\theta
\\&\lesssim 
\int_{S^n} \int_0^{R(\theta)} (u^2 + |u| |r\p_r u|) r^{n} \d r \d\theta
\lesssim \norm{u}_{K}^2 + \norm{h\nabla u}_{K}^2.
    \end{align*}
\end{proof} 

\begin{proof}[Proof of Lemma \ref{lem_trace_bd}]
We choose an extension $\hat u$ of $u$ so that $\norm{\hat u}_{H^k(\R^{1+n})} \lesssim \norm{u}_{H^k(M)}$.
Let $u_h \in \hat V_h^p$ be the Scott--Zhang interpolation of $\hat u$
where 
    \begin{align*}
\hat V^p_{h} = \{u \in H^1(M_h) : u |_K \in \mathbb{P}_p(K)
\text{ for all $K \in \hat{\mathcal T}_h$} \}.
    \end{align*}
Clearly $u_h|_M \in V^p_h$, and the now classical result \cite{SZ90} says that the analogue 
    \begin{align*}
\norm{\hat u- u_h}_{H^k(\hat \T)}
\lesssim h^k \norm{\hat u}_{H^k(\R^{1+n})}
    \end{align*}
of (\ref{interp2}) holds. This implies (\ref{interp2}) since $\T$ is obtained from $\hat \T$ via a restriction.
\end{proof}

\bibliographystyle{abbrv}
\bibliography{refs}
\ifoptionfinal{}{
}
\end{document}